\documentclass[a4paper,11pt,reqno]{amsart} 
\usepackage{amsmath}
\usepackage{amsfonts}
\usepackage[T1]{fontenc}
\usepackage[utf8]{inputenc}
\usepackage[mathscr]{euscript}
\usepackage{verbatim}
\usepackage{amssymb,latexsym}
\usepackage{amsthm}
\usepackage{eufrak}
\usepackage{color}

\usepackage[lmargin=2.5 cm,rmargin=2.5 cm,tmargin=3.5cm,bmargin=2.5cm,paper=a4paper]{geometry}
\DeclareMathOperator{\curl}{curl}

\newtheorem{thm}{Theorem}[section]
\newtheorem{pro}[thm]{Proposition}

\newtheorem{lem}[thm]{Lemma}

\newtheorem{definition}[thm]{Definition}

\newtheorem{proposition}[thm]{Proposition}
\newtheorem{corollary}[thm]{Corollary}

\theoremstyle{remark}
\newtheorem{rem}[thm]{Remark}

\newcommand{\Int}[2]{\displaystyle{\int_{#1}^{#2}}}

\newcommand{\Ab}{\mathbf {A}}

\newcommand{\R}{\mathbb{R}}

\newcommand{\Hr}{\mathcal H_{{\rm harm}}(\zeta,\xi)}
\newcommand{\Hw}{\mathcal H_{\zeta,\beta,\xi,h}}
\newcommand{\Hwo}{\mathcal H_{\zeta,0,\xi,h}}

\newcommand{\norm}[1]{\left\|#1\right\|}

\def\sig#1{\vbox{\hsize=5.5cm
\kern2cm\hrule\kern1ex
\hbox to \hsize{\strut\hfil #1 \hfil}}}
\newcommand\signatures[4]{%
\vspace{3cm}
\hbox to \hsize{\hfil #1, \today\hfil}
\vspace{3cm}
\hbox to \hsize{\quad#2\hfil\hfil #3\quad}
\vspace{3cm}
\hbox to \hsize{\hfil#4\hfil}}
\numberwithin{equation}{subsection}

\title[Diamagnetism vs Robin condition]{Diamagnetism versus Robin condition and concentration of ground states}
\author[A. Kachmar]{Ayman Kachmar}
\address[A. Kachmar]{Lebanese University, Department of Mathematics, Hadath, Lebanon}
\email{ayman.kashmar@gmail.com}
\date{\today}
\begin{document}
\maketitle
\begin{abstract}
We estimate the ground state energy for the magnetic Laplacian with   a Robin condition. 
In a  special asymptotic limit, we find that the magnetic field does not contribute to the two-term expansion of the ground state energy, thereby proving that the Robin condition weakens diamagnetism. We discuss a semi-classical version of the operator and prove that the ground states concentrate near the boundary points of  maximal curvature.  
\end{abstract}

\section{Introduction}

This paper is motivated by two different questions. The first question concerns  the analysis of the ground state energy and the concentration of the ground states for the magnetic Laplacian with a Robin condition and a semi-classical parameter, and is a continuation of the work in \cite{Ka-cr, Ka-jmp}.  The second question is around the influence of the Robin condition on diamagnetism. We will find that these two questions are intimately related and we will get satisfactory answers for both. Besides the concentration of the ground states near the points of maximal curvature, we will identify the optimal strength of the Robin condition/magnetic field such that diamagnetism occurs to leading order of the energy. 

The results in this paper are valid in an open set   $\Omega\subset\R^{2}$. We will assume that the boundary $\Gamma=\partial\Omega$ of $\Omega$ is $C^3$  smooth, 
compact and consists  of a finite number of connected
components. Our assumptions allow for $\Omega$  to be an {\it interior}
or {\it exterior} domain, and the smoothness of the boundary ensures the existence of a normal vector every where on the boundary. We will denote by $\nu$ the unit outward normal vector (field) of  $\partial\Omega$.

\subsection{Concentration of ground states}

Let us  introduce the  semi-classical magnetic Laplacian that we will study. 
Consider the  magnetic potential:
\begin{equation}\label{eq:A0}
\R^2\ni(x_1,x_2)\mapsto\Ab_0(x_1,x_2)=(-x_2,0)\,.
\end{equation}
This magnetic potential generates the constant magnetic field:
\begin{equation}\label{eq:mf}
B:= {\rm curl}\, \Ab_0=1\,.
\end{equation}
We are interested in the same magnetic Laplacian studied in \cite{Ka-jmp} which involves four  parameters, the strength of the magnetic field $b>0$, the {\it semi-classical} parameter $h>0$, two   parameters $\gamma\in\R\setminus\{0\}$ and $\alpha\in\R$ that will serve in defining the boundary condition.     
The operator is
\begin{equation}\label{Shr-op-Gen}
\mathcal{L}^{\alpha,\gamma}_{h,b,\Omega}=-(h\nabla-ib\Ab_0)^{2}\quad {\rm in~}L^2(\Omega),
\end{equation}
with a boundary condition of the third type (Robin condition)
\begin{equation}\label{eq:bc}
\nu\cdot(h\nabla-i\Ab_0)u+h^\alpha\gamma\,u=0\quad{\rm on}~\partial\Omega\,.
\end{equation}
This operator can be defined via  Friedrich's Theorem and the closed
semi-bounded quadratic form,
\begin{equation}\label{QF-Gen}
 \mathcal{Q}^{\alpha,\gamma}_{h,b,\Omega}(u):=\norm{(h\nabla-ib\Ab_0)u}^{2}_{L^{2}(\Omega)}+h^{1+\alpha}\gamma\Int{\partial\Omega}{}|u(x)|^{2}dx\,.
\end{equation}
This quadratic form is defined in the `magnetic' Sobolev space
\begin{equation}\label{eq:H1-mag}
H^1_{h^{-1}b\Ab_0}(\Omega)=\{u\in L^2(\Omega)~:~(\nabla-ih^{-1}b\Ab_0)\in L^2(\Omega)\}\,.\end{equation}
The  parameters $\alpha$ and $\gamma$ serve in controlling the `strength' of the boundary condition in  \eqref{QF-Gen}. As we shall see, the {\it sign} of $\gamma$ and the {\it values} of $\alpha$ have a strong influence on the spectrum of the magnetic Laplacian in \eqref{Shr-op-Gen}. Notice that  $\gamma=0$ corresponds to the extensively studied magnetic Laplacian with Neumann condition (cf. \cite{HM, FH-b}), while $b=0$ corresponds to the Laplacian {\it without} a magnetic field. That justifies the assumption $\gamma\not=0$ and $b>0$.

The ground state energy (lowest eigenvalue) of the operator in \eqref{Shr-op-Gen} is:
\begin{equation}\label{eq:gse}
\mu_1(h;b,\alpha,\gamma)=\inf_{u\in H^1_{b\Ab_0}(\Omega)\setminus\{0\}}\frac{\mathcal{Q}^{\alpha,\gamma}_{h,b,\Omega}(u)}{\|u\|^2_{L^2(\Omega)}}\,.
\end{equation}
We will study the asymptotic limit where the semi-classical parameter $h$ tends to $0_+$, while the parameters $b$, $\alpha$ and $\gamma$ are assumed fixed. In this regime, we can reduce to the case $b=1$ simply by observing that, for all $b>0$,
\begin{equation}\label{eq:b=1}
\mu_1(h;b,\alpha,\gamma)=b^2\mu_1\big(b^{-1}h;b=1,\alpha,b^{-1+\alpha}\gamma\big)\,,
\end{equation}
and that as long as we assume $b$ fixed, $b^{-1}h\to0$ when $h\to0_+$.

The results in \cite{Ka-jmp} distinguish between two situations. The first one corresponds to $\alpha\geq 1/2$ and is fairly understood: A two term asymptotic expansion of the ground state energy in \eqref{eq:gse} is established; the ground state energy is in the discrete spectrum (cf. \cite{KP}); and the ground states are localized near the boundary points where the curvature is maximal.

The second situation corresponds to $\alpha<\frac12$ and is less understood.  Here the sign of $\gamma$ will play a dominant role. The contribution in this paper will clarify the situation when $\gamma<0$. For $\gamma>0$, the ground state energy satisfies
$$\mu_1(h;b,\alpha,\gamma)=bh+ho(1)\quad (h\to0_+)\,.$$
For $\gamma<0$, the behavior of the ground state energy is completely different and  displayed as follows,
$$\mu_1(h;b,\alpha,\gamma)=-\gamma^2h^{2\alpha}+h^{2\alpha}o(1)\quad (h\to0_+)\,.$$
Note that this asymptotic expansion does not involve the strength of the magnetic field $b$. Again,  the ground state energy is an eigenvalue, as long as the semi-classical parameter $h$ is sufficiently small (cf. \cite{KP}). In \cite{Ka-jmp}, it is proved that the ground states concentrate near the boundary (when $\alpha<\frac12$ and $\gamma<0$). The natural question is then weather one can refine the concentration near some special boundary points, e.g. points of maximal curvature. We will give an affirmative answer to this question in Theorem~\ref{thm:gs-gse} below.

Since the boundary is assumed smooth, there exists a geometric constant $t_0\in(0,1)$ such that, if ${\rm dist}(x,\partial\Omega)<t_0$, then we may assign a unique point $p(x)\in\partial\Omega$ such that ${\rm dist}(x,p(x))={\rm dist}(x,\partial\Omega)$. The function $\kappa(\cdot)$ denotes the curvature along the boundary.

In the statement of Theorem~\ref{thm:gs-gse}, $\kappa_{\max}$ is the maximum of the curvature along the boundary,
$$\zeta=\frac{b}{\gamma^2}h^{1-2\alpha}\,,$$
$n\in\mathbb N$ is the smallest positive integer satisfying
\begin{equation}\label{eq:cond-n}
(n+1)\frac{1-2\alpha}{1-\alpha}>\frac12\,,\end{equation}
and $e_n(\zeta)$ is the quantity that we will introduce in \eqref{eq:e-n} below. As $\zeta\to0_+$,  $e_n(\zeta)$ behaves like $\frac14\zeta^2$ (cf. Remark~\ref{rem:en}).   

Now we are ready to state:

\begin{thm}\label{thm:gs-gse}
Suppose that $b>0$, $\alpha<\frac12$ and $\gamma<0$. 
\begin{enumerate}
\item As  $h\to0_+$, the ground state energy in \eqref{eq:gse} satisfies
$$\mu_1(h;b,\alpha,\gamma) =
\left\{
\begin{array}{ll}
-\gamma^2h^{2\alpha}+e_n(\zeta)\gamma^2h^{2\alpha} +\gamma\kappa_{\max}h^{1+\alpha}+h^{1+\alpha}o(1)&{\rm if~}\frac13<\alpha<\frac12\,,\\
&\\
-\gamma^2h^{2/3}+(\frac1{4\gamma^2}b^2+\gamma\kappa_{\max})h^{4/3}+ h^{4/3}o(1)&{\rm if~}\alpha=\frac13\,,\\
&\\
-\gamma^2h^{2\alpha} +\gamma\kappa_{\max}h^{1+\alpha}+h^{1+\alpha}o(1)&{\rm if~}\alpha<\frac13\,.
\end{array}\right.
$$
\item There exist  constants $\rho\in(0,\frac12)$, $\eta^*\in(0,\frac14)$, $C>0$ and $h_0>0$ such that, for all $h\in (0,h_0)$, every ground state $u_h$ of \eqref{eq:gse} satisfies
$$\int_{\Omega_{\rm bnd}}|u_{h,\zeta}|^2\,dx\leq C\exp\left(-h^{\frac{\eta^*-\frac14}{2(1-\alpha)}}\right)\,,\quad \int_{\Omega_{\rm int}}|u_{h,\zeta}|^2\,dx\leq C\exp\left(
-h^{\frac{\rho-\frac12}{2(1-\alpha)}}\right)\,,$$
where
$$
\begin{aligned}
&\Omega_{\rm int}=\{x\in\Omega~:~{\rm dist}(x,\partial\Omega)\geq h^{\frac\rho{2(1-\alpha)}}\}\quad{\rm and}\\
&\Omega_{\rm bnd}=\{x\in\Omega\setminus\overline\Omega_{\rm int}~:~\kappa_{\max}-\kappa(p(x))\geq h^{\frac{\eta^*}{2(1-\alpha)}}\}\,.
\end{aligned}
$$
\end{enumerate}
\end{thm}   

Notice that the asymptotic expansions for $\mu_1(h;b,\alpha,\gamma)$ are compatible in the cases $\alpha=\frac13$ and $\alpha<\frac13$. Formally, we get the expansion for $\alpha<\frac13$ by taking $\gamma\to-\infty$ in the case $\alpha=\frac13$. 

Theorem~\ref{thm:gs-gse} adds two improvements to the results in \cite{Ka-jmp} by
\begin{itemize}
\item establishing  a two-term expansion for the ground state energy;
\item refining the concentration of the ground states near the points of maximal curvature.
\end{itemize} 

\begin{rem}
The magnetic field is assumed constant in Theorem~\ref{thm:gs-gse}, but the methods in the this paper should allow for dealing with non-constant $C^1$ magnetic fields.
\end{rem}

\subsection{Diamagnetism}

Here we will discuss the question of diamagnetism. We will find that imposing a Robin condition may slow diamagnetism (and even neglect this effect). Let $\beta\in\R$, $H\geq 0$ and  $\mathcal L^\beta(H)$ be the self-adjoint operator in $L^2(\Omega)$,
\begin{equation}\label{L(H)}
\mathcal L^\beta(H)=-(\nabla-iH\Ab_0)^2
\end{equation}
with domain
$$D\big(\mathcal L^\beta(H)\big)=\{u\in H^1_{H\Ab_0}(\Omega)~:~(\nabla-iH\Ab_0)^2\in L^2(\Omega)~{\rm and}~\nu\cdot(\nabla -iH\Ab_0)u+\beta u=0~{\rm on~}\partial\Omega\}\,.
$$
Here the magnetic Sobolev space $H^1_{H\Ab_0}(\Omega)$ is introduced in \eqref{eq:H1-mag}. Note that, for $H=0$, $\mathcal L^\beta(0)$ is the Robin Laplacian, while for $H>0$, $\mathcal L^\beta(H)$ is the magnetic Laplacian with a (magnetic) Robin condition. Let $\sigma\big(\mathcal L^\beta(H)\big)$ be the spectrum of the operator $\mathcal L^\beta(H)$. We introduce the ground state energy,
\begin{equation}\label{eq:gse*}
\tilde\mu_1(\beta;H)=\inf\sigma\big(\mathcal L^\beta(H)\big)\,.
\end{equation}
The diamagnetic inequality yields, for all $\beta\in\R$ and $H>0$,
\begin{equation}\label{eq:diamag}\tilde\mu_1(\beta;H) -\tilde\mu_1(\beta;0)\geq0\,.\end{equation}
In physical terms, this inequality refers to diamagnetism. It simply says that introducing a magnetic field increases the ground state energy. 
We will see that, when $\beta\to-\infty$, diamagnetism is weak in the sense that the difference $ \tilde\mu_1(\beta;H) -\tilde\mu_1(\beta;0)$ is small. This property is a unique feature for the Robin condition as it fails for the Neumann and Dirichelt conditions. On the contrary, in simply connected domains, a Neumann boundary condition induces  {\it strong} diamagnetism
 (cf. \cite{FH-b}). 

The asymptotic analysis of the spectrum of the Robin Laplacian is studied in many papers,
cf. \cite{Exner, HeKa, HP, Pan, PanP, PanP2}. In particular,  as $\beta\to-\infty$, the ground state energy satisfies,
\begin{equation}\label{eq:L(0)}
\tilde\mu_1(\beta;0)=-\beta^2+\beta\kappa_{\max}+\beta o(1)\,.
\end{equation}
We will write an  asymptotic expansion for the magnetic ground state energy $\tilde\mu_1(\beta;H)$ valid when $\beta\to-\infty$ and $H\to\infty$ simultaneously. This is the content of Theorem~\ref{thm:dia-mag} below. In particular, we will get a fair knowledge about the difference in \eqref{eq:diamag} which measures the strength of diamagnetism. 

The statement of Theorem~\ref{thm:dia-mag} requires a real-valued function $\Theta(\cdot)$   introduced in \cite{Ka-jmp}. For $\gamma\in\R$, define the ground state energy
\begin{equation}\label{eq:Theta}
\Theta(\gamma)=\inf_{\xi\in\R}\left(\inf_{u\in B^1(\R_+)\setminus\{0\}}
\frac{\displaystyle\int_0^\infty\big(|u'(t)|^2+|(t-\xi)u|^2\big)\,dt-\gamma |u(0)|^2}{\|u\|_{L^2(\R_+)}^2}
\right)\,,
\end{equation}
where
$$B^1(\R_+)=\{u\in L^2(\R_+)~:~u',(t-\xi)u\in L^2(\R_+)\}\,.$$

In \cite{Ka-jmp}, it is proved that $\Theta(\cdot)$ is smooth, increasing and $\Theta(\gamma)>-\gamma^2$ for all $\gamma\leq0$.

Now we are ready to state:

\begin{thm}\label{thm:dia-mag}
Let $\alpha\in\R\setminus\{1\}$, $0<c_1<c_2$ and $\beta_0<0$. Suppose that
$$\beta<\beta_0\quad{\rm and}\quad c_1|\beta|^{\frac1{1-\alpha}}\leq H\leq c_2|\beta|^{\frac1{1-\alpha}}\,.$$
\begin{enumerate}
\item If $\alpha>\frac12$,   the ground state energy satisfies, as $\beta\to-\infty$,
$$\tilde\mu_1(\beta;H)=\Theta(0)H+Ho(1)\,.$$
where $\Theta_0\in(0,1)$ is a universal constant.
\item If $\alpha=\frac12$,  the ground state energy satisfies, as $\beta\to-\infty$,
$$\tilde\mu_1(\beta;H)=H\Theta_0(\beta H^{-1/2})+H o(1)\,.$$
\item If $\frac13<\alpha<\frac12$,  the ground state energy satisfies, as $\beta\to-\infty$,
$$\tilde\mu_1(\beta;H)=-\beta^2+e_n\big(H\beta^{-2}\big)\beta^2+\beta\kappa_{\max}+\beta o(1)\,,$$
where $n$ is the smallest positive integer satisfying \eqref{eq:cond-n} and $e_n(\cdot)$ is introduced in \eqref{eq:e-n}.
\item If $\alpha=\frac13$,  the ground state energy satisfies, as $\beta\to-\infty$,
$$\tilde\mu_1(\beta;H)=-\beta^2+\left(\frac{H^2\beta^{-3}}4+\kappa_{\max}\right)\beta
+\beta o(1)\,.$$
\item If $\alpha<\frac13$,   the ground state energy satisfies, as $\beta\to-\infty$,
$$\tilde\mu_1(\beta;H)=-\beta^2+\kappa_{\max}\beta+\beta o(1)\,.$$
\end{enumerate}
\end{thm}

Theorem~\ref{thm:dia-mag} suggests that, in the limit $\beta\to-\infty$, diamagnetism occurs to leading order when  the strength of the magnetic field satisfies $H\approx \beta^{\sigma}$
and $\sigma\geq 2$.  

In the situation where $H\approx \beta^\sigma$ and $\sigma<2$, diamagnetism occurs as a correction term and will compete with the correction term coming from the curvature of the boundary. According to Theorem~\ref{thm:dia-mag}:
\begin{itemize}
\item If $\frac32<\sigma<2$, then diamagnetism  occurs in the second correction term while the influence of the curvature occurs in the third correction term\,;
\item If $\sigma=\frac32$, both diamagnetism and curvature corrections appear in the second correction term\,;
\item If $\sigma<\frac32$, dia-magnetism is weak and its contribution is negligible compared to the contribution of the  curvature correction term (compare with \eqref{eq:L(0)}).  
\end{itemize}

Through this paper, the following notation will be used. $C$ denotes a constant independent from the semi-classical parameter $h$. $\mathcal O(h^\infty)$ is a quantity satisfying that, for all $N\in\mathbb N$, there exist two constants $h_0\in(0,1)$ and $C_N>0$ such that, for all $h\in (0,h_0)$,
$|\mathcal O(h^\infty) |\leq C_nh^N$.

The paper is organized as follows. In Section~\ref{sec:AOp}, we analyze three auxiliary differential operators useful to prove Theorem~\ref{thm:gs-gse}. The proof of Theorem~\ref{thm:gs-gse} occupies all of Section~\ref{sec:proof}. Finally, in Section~\ref{sec:dia-mag}, we explain how to get the result in Theorem~\ref{thm:dia-mag} from the existing results on the semi-classical magnetic Laplacian with a Robin condition, in particular those in Theorem~\ref{thm:gs-gse}.

\section{Analysis of auxiliary operators} \label{sec:AOp}

\subsection{1D Laplacian on the half line}\
Here we introduce a simple $1D$ operator that will play a fundamental role in the next sections. This operator arises naturally in the analysis of the Robin Laplacian without magnetic field (cf. \cite{Pan, HeKa}).
The operator is
\begin{equation}\label{defH00}
\mathcal H_{0,0}:=-\frac{d^2}{d\tau^2} \mbox{ in }  L^2(\R_+)
\end{equation}
with domain
\begin{equation}
\{u\in
H^2(\R_+)~:\,u'(0)=- u(0)\}\,.
\end{equation}
 The spectrum of this operator is
$\{-1\}\cup[0,\infty)$,
and $-1$ is a simple eigenvalue with the $L^2$ normalized eigenfunction
\begin{equation}\label{eq:ef-u0}
u_0(\tau)=\sqrt{2}\,\exp(-\tau)\,.
\end{equation}


\subsection{Harmonic oscillator on the half-line}

The key element in the proof of Theorem~\ref{thm:mag-Lap} is the analysis of the harmonic oscillator
\begin{equation}\label{eq:Hr}
\Hr=-\frac{d^2}{d\tau^2}+(\zeta t-\xi)^2\quad{\rm in~}L^2(\R_+)\,,
\end{equation}
with domain
$$D\big(\Hr\big)=\{u\in H^1(\R_+)~:~\tau^ku\in L^2(\R_+),~k\in\{1,2\},~u'(0)=-u(0)\}\,.$$
Here $\zeta>0$ and $\xi\in\R$ are two parameters. Let us denote by $(\lambda_n(\Hr))$ the increasing sequence of the eigenvalues of $\Hr$ counting multiplicities. We will study the aymptotic behavior of the eigenvalue
\begin{equation}\label{eq:ev-Hr}
\lambda_1(\Hr)=\inf\Big(\sigma\big(\Hr\big)\Big)\,,
\end{equation}  
as $\zeta\to0$.

By comparison with the operator in \eqref{defH00}, we get:

\begin{lem}\label{lem:Hr*}
For all $\zeta>0$ and $\xi\in\R$, it holds,
$$\lambda_1(\Hr)\geq -1\quad{\rm and}\quad\lambda_2(\Hr)\geq 0\,.$$
\end{lem}

The lower bound in Lemma~\ref{lem:Hr*} can be improved as follows:

\begin{lem}\label{lem:Hr}
There exists a universal constant $A_0>0$ such that, if $0<\zeta<1$ and $|\xi|\geq A_0\zeta$, then
$$\lambda_1(\Hr)\geq -1+\frac32\zeta^2\,.$$
\end{lem}
\begin{proof}
Let $u$ be an $L^2$ normalized ground state of the operator $\Hr$. Let us write,
\begin{align*}
\lambda_1(\Hr)&=\int_0^\infty\Big(|u'|^2+|(\zeta\tau-\xi)u|^2\Big)\,d\tau-|u(0)|^2\\
&=(1-\zeta^2)\left(\int_0^\infty|u'|^2\,d\tau-|u(0)|^2\right)+\zeta^2
\left(\int_0^\infty\big(|u'|^2+|(\tau-\zeta^{-1}\xi)u|^2\,d\tau-|u(0)|^2\right)\,.
\end{align*} 
We know that the lowest eigenvalue of the operator in \eqref{defH00} is $-1$. Let $\mu(A)$ be the eigenvalue of the operator
$$-\frac{d^2}{d\tau^2}+(\tau-A)^2\quad{\rm in~}L^2(\R_+)\,,$$
with the boundary condition $u'(0)=-u(0)$. 

Now, it results from the min-max principle that
$$\lambda_1(\Hr)\geq -(1-\zeta^2)+\zeta^2\mu(\zeta^{-1}\xi)\,.$$
When $|\zeta^{-1}\xi|$ is sufficiently large, we get
$\mu(\zeta^{-1}\xi)\geq \frac12$, which is a result of the following two facts proved in \cite{Ka-jmp},
$$\lim_{A\to-\infty}\mu(A)=\infty\quad{\rm and}\quad \lim_{A\to\infty}\mu(A)=1\,.$$
\end{proof}

We will prove that:

\begin{thm}\label{thm:Hr}
Let  $n\in\mathbb N$. 
There exist $C>0$, a collection of vectors
$$\Big\{\mu_j=(\mu_{j,1},\mu_{j,2},\cdots,\mu_{j,2j+1})\in\R^{2j+1}\Big\}_{j=1}^n\,,$$
and a collection of vector functions,
 $$\Big\{u_j=(u_{j,1},u_{j,2},\cdots,u_{j,2j+1})\in \big(\mathcal S (\R_+)\big)^{2j+1}\Big\}_{j=1}^n\,,$$ 
 such that, if 
 $0<\zeta\leq 1$ and $|\xi|\leq 1$, then
$$\|(\Hr-\lambda_n)w_n\|_{L^2(\R_+)}\leq C\Big(\zeta^{2n+2}+\xi^{2n+2}\Big)\,,$$
where
$$w_n=u_0+\sum_{j=1}^n\sum_{p=0}^{2j}\zeta^{2j-p}\xi^pu_{j,p+1}\,,$$
$u_0$ is  the eigenfunction in \eqref{eq:ef-u0},
and
$$
\lambda_n=-1
+\sum_{j=1}^n\sum_{p=0}^{2j}\mu_{j,p+1}\zeta^{2j-p}\xi^{p}\,.
$$
Furthermore,
$$\mu_1=\Big(\mu_{1,1}=\frac12,\mu_{1,2}=-1  ,\mu_{1,3}=1\Big)\,.$$
\end{thm}
\begin{proof}\

{\bf Step~1: Construction of $(\mu_1,u_1)$.}

Here we construct  $\mu_1=(\mu_{1,1},\mu_{1,2},\mu_{1,3})\in\R^3$  and $u_1=(u_{1,1},u_{1,2},u_{1,3})$ such that the conclusion of Theorem~\ref{thm:Hr} is valid for $n=1$.

Let us define
$$\lambda_1=-1+\mu_{1,1}\zeta^2+\mu_{1,2}\zeta\xi+\mu_{1,3}\xi^2\quad{\rm and}\quad
w_1(\tau)=u_0(\tau)+\zeta^2u_{1,1}(\tau)+\zeta\xi u_{1,2}(\tau)+\xi^2 u_{1,3}(\tau).$$
For simplicity of the notation, we will write $H=\Hr$. Notice that, since $\left(-\frac{d^2}{d\tau^2}+1\right)u_0=0$,
\begin{multline}\label{eq:step1}
(H-\lambda_1)w_1=
\zeta^2\left[\left(-\frac{d^2}{d\tau^2}+1\right)u_{1,1}+(\tau^2-\mu_{1,1})u_0\right]
+\zeta\xi\left[\left(-\frac{d^2}{d\tau^2}+1\right)u_{1,2}-(2\tau+\mu_{1,2})u_0\right]\\
+\xi^2\left[\left(-\frac{d^2}{d\tau^2}+1\right)u_{1,3}+(1-\mu_{1,3})u_0\right]+R_1\,.
\end{multline}
The remainder $R_1$ is 
\begin{equation}\label{eq:Rstep1}
\begin{aligned}
R_1=&\zeta^4(\tau^2-\mu_{1,1})u_{1,1}+\zeta^3\xi\big[(\tau^2-\mu_{1,1})u_{1,2}+
(\mu_{1,2}-2\tau)u_{1,2}\big]\\
&+\zeta^2\xi^2\big[(1-\mu_{1,1})u_{1,3}+(\mu_{1,2}-2\tau)u_{1,2}+(1-\mu_{1,3})u_{1,1}\big]\\
&+\zeta\xi^3\big[(1-\mu_{1,3})u_{1,2}+(\mu_{1,2}-2\tau)u_{1,3}\big]
+\xi^4(1-\mu_{1,3})u_{1,3}\,.
\end{aligned}
\end{equation}
We choose the coefficients and the functions in \eqref{eq:step1} so that all the terms on the left side vanish. This is possible since the operator $-\frac{d^2}{d\tau^2}+1$ can be inverted in the orthogonal complement of the eigenfunction $u_0$. That way we choose,
\begin{align*}
&\mu_{1,1}=\int_0^\infty \tau^2|u_0(\tau)|^2\,d\tau=\frac12\,,\quad u_{1,1}=-\left(-\frac{d^2}{d\tau^2}+1\right)^{-1}\{(\tau^2-\mu_{1,1})u_0\}\\
&\mu_{1,2}=-\int_0^\infty 2\tau|u_0(\tau)|^2\,d\tau=-1\,,\quad u_{1,2}=\left(-\frac{d^2}{d\tau^2}+1\right)^{-1}\{(2\tau+\mu_{1,1})u_0\}\\
&\mu_{1,3}=\int_0^\infty |u_0(\tau)|^2\,d\tau=1\,,\quad u_{1,3}=-\left(-\frac{d^2}{d\tau^2}+1\right)^{-1}\{(1-\mu_{1,3})u_0\}\,.
\end{align*}
The operator $\left(-\frac{d^2}{d\tau^2}+1\right)^{-1}$ respects the Schwartz space $\mathcal S(\R_+)$. The proof of this is standard  (cf. \cite[Lemma~A.5]{FH}). Now, since $|\xi|\leq A\zeta$, we infer from  \eqref{eq:step1}
and \eqref{eq:Rstep1},
$$\|(H-\lambda_1)w_1\|_{L^2(\R_+)}\leq C(\zeta^4+\xi^4)\,.$$

{\bf Step~2: The iteration process.}

Suppose that we have constructed $(\mu_j)_{j=1}^n$ and $(u_j)_{j=1}^n$ such that
$$(H-\lambda_n)w_n=R_n+f_{n,\zeta,\xi}\,,$$
$R_n$ has the form
$$R_n=\sum_{p=0}^{2n+2}\zeta^{2n+2-p}\xi^p v_{n,p}\,,$$
for a collection $(v_{n,p})$ of Schwartz functions  that do not depend on $\zeta$ and $\xi$,
and the function $f_{n,\zeta,\xi}$ satisfies,
$$\|f_{n,\zeta,\xi}\|_{L^2(\R_+)}\leq C\big(\zeta^{2n+4}+\xi^{2n+4}\big)\,,$$
where $C$ is a constant independent of $\zeta$ and $\xi$.

We outline  the construction of
$$\mu_{n+1}=(\mu_{n+1,1},\mu_{n+1,2},\cdots,\mu_{n+1,2n+3})\in\R^{2n+3}\quad{\rm and}\quad
u_{n+1}=(u_{n+1,1},u_{n+1,2},\cdots,u_{n+1,2n+3})\,,$$
such that
$$(H-\lambda_{n+1})w_{n+1}=R_{n+1}+f_{n+1,\zeta,\xi}\,,$$
$R_{n+1}$ has the form
$$
R_{n+1}=\sum_{p=0}^{2n+4}\zeta^{2n+4-p}\xi^p v_{n+1,p}\,,$$
for a collection $(v_{n+1,p})$ of Schwartz functions  that do not depend on $\zeta$ and $\xi$,
and $f_{n+1,\zeta,\xi}$ satisfies,
$$\|f_{n+1,\zeta,\xi}\|_{L^2(\R_+)}\leq C\big(\zeta^{2n+6}+\xi^{2n+6}\big)\,,$$
where $C$ is a constant independent of $\zeta$ and $\xi$.

We expand $(H-\lambda_{n+1})w_{n+1}$ and rearrange the terms in the form,
\begin{align*}
(H-\lambda_{n+1})w_{n+1}=&\sum_{p=0}^{2n+2}\zeta^{2n+2-p}\xi^p
\left\{\left(-\frac{d^2}{d\tau^2}+1\right)u_{n+1,p+1}+v_{n,p}-\mu_{n+1,p+1}u_0\right\}\\
&+\sum_{p=0}^{2n+4}\zeta^{2n+4-p}\xi^p v_{n+1,p}+\sum_{p=0}^{2n+6}\zeta^{2n+6-p}\xi^pg_{n+1,p}\,,
\end{align*}
where the functions $v_{n+1,p}$ and $g_{n+1,p}$ are expressed in terms of the functions
$u_{j,q}$ and the real numbers $\mu_{j,q}$.

All what we have to do now is  to select the functions $u_{n+1,p+1}$ and the real numbers $\mu_{n+1,p+1}$ such that
$$\sum_{p=0}^{2n+2}\zeta^{2n+2-p}\xi^p
\left\{\left(-\frac{d^2}{d\tau^2}+1\right)u_{n+1,p+1}+v_{n,p}-\mu_{n+1,p+1}u_0\right\}=0\,.
$$
To that end, we select $\mu_{n+1,p+1}$ such that,
$$\mu_{n+1,p+1}=\int_0^\infty v_{n,p}\,u_0\,d\tau\,,$$
so that
$$v_{n,p}-\mu_{n+1,p+1}u_0~\perp~ u_0\quad{\rm in ~}L^2(\R_+)\,.$$
Finally, we define the function $u_{n+1,p+1}$ as follows,
$$u_{n+1,p+1}=-\left(-\frac{d^2}{d\tau^2}+1\right)^{-1}\big(v_{n,p}-\mu_{n+1,p+1}u_0\big)\,.$$
\end{proof}

As a consequence of Theorem~\ref{thm:Hr}, Lemma~\ref{lem:Hr*} and the spectral theorem, 
we get:
\begin{corollary}\label{corl:Hr}
Let $n\in\mathbb N$ and $A>0$. If $|\xi|\leq A\zeta$, then as $\zeta\to0_+$, the eigenvalue $\lambda_1(\Hr)$ satisfies,
$$\lambda_1(\Hr)=-1+ \sum_{j=1}^n\sum_{p=0}^{2j}\mu_{j,p+1}\zeta^{2j-p}\xi^{p} +\mathcal O(\zeta^{2n+2}).$$
\end{corollary}

\begin{definition}\label{eq:defn:en}
Let $n\in\mathbb N$ and $A_0$ be the universal constant in Lemma~\ref{lem:Hr}. For all $\zeta\in(0,1)$, we define the following quantity
\begin{equation}\label{eq:e-n}
e_n(\zeta)=\inf\{ f_n(\zeta,\xi)~:~{|\xi|\leq \zeta\max(A_0,1)}\}\,,
\end{equation} 
where
\begin{equation}\label{eq:f-n}
f_n(\zeta,\xi)=\sum_{j=2}^n\sum_{p=0}^{2n}\zeta^{2n-p}\xi^p\mu_{j,p+1}\,,
\end{equation}
and $(\mu_{j,p+1})$ are the constants in Theorem~\ref{thm:Hr}.
\end{definition}

\begin{rem}\label{rem:en}
Note that for $n=1$, $f_1(\zeta,\xi)=\frac12\zeta^2-\zeta\xi+\xi^2$, and
$$\min f_1(\zeta,\xi)=f_1\Big(\zeta,\frac12\zeta\Big)=\frac14\zeta^2\,.$$
Consequently, for all $n\in\mathbb N$, as $\zeta\to0_+$, 
$$e_n(\zeta)=\frac14\zeta^2+\mathcal O(\zeta^4)\,.$$
\end{rem}

\subsection{A family of operators in a weighted space}\label{sec:Hw}

Here we will study an operator that arises in many papers concerned with the semi-classical magnetic Laplacian (cf. \cite{HM, Ka-jmp}). 
Let $h\in(0,1)$, $\zeta\in(0,1)$, $\xi\in\R$, $\beta\in\R$, $\delta\in(0,1/2)$, $m\geq 0$, $\sigma\in(0,1)$, $M>0$ and
$|\beta|h^{\frac12-\delta}<\frac13$.

Consider the Hilbert-space
$$L^2\Big((0,h^{-\delta}),\tilde a\,d\tau\Big)\,,\quad \tilde a=1-(\beta+mh^\sigma )h^{1/2}\tau\,,\quad \|\cdot\|_{L^2\big((0,h^{-\delta}),\tilde a\,d\tau\big)}=\left(\int_{0}^{h^{-\delta}}|\cdot|^2\,\tilde a\,d\tau\right)^{1/2}\,,$$
and the self-adjoint
operator
\begin{equation}\label{eq:H0b}
\begin{aligned}
\Hw=&-\tilde a^{-1}\partial_\tau(\tilde a\partial_\tau)+(1+h^{1/2}\Delta_{\beta,\tau})\Big(\zeta\tau(1-\frac12\beta h^{1/2}\tau)-\xi\Big)^2\\
=&
-\frac{d^2}{d\tau^2}+(\zeta\tau-\xi)^2+(\beta+mh^\sigma) h^{1/2}\big(1-(\beta+mh^\sigma)
h^{1/2}\tau\big)^{-1}\frac{d}{d\tau}\\
&+h^{1/2}\Delta_{\beta,\tau}(\zeta\tau-\xi)^2+\beta h^{1/2}\zeta\tau^2(1+h^{1/2}\Delta_{\beta,\tau})\Big(-(\zeta \tau-\xi)+\frac14\beta h^{1/2}\zeta\tau^2\Big)
\,,
\end{aligned}
\end{equation}
in $L^2\Big((0,h^{-\delta}),\tilde a\,d\tau\Big)$. Here $\Delta_{\beta,\tau}$ is a function of $(\beta,\tau)$ and satisfies, for all $h\in(0,1)$,
$$|\Delta_{\beta,\tau}|\leq M(\beta+1)\tau\,.$$
The  domain of the operator $\Hw$ is
\begin{equation}\label{eq:domH0b}
D\Big(\Hw\Big)=\{u\in H^2((0,h^{-\delta}))~:~u'(0)=-u(0)\quad{\rm and}\quad u(h^{-\delta})=0\}\,.
\end{equation}
The operator $\Hw$ is the Friedrichs extension in
$L^2\Big((0,h^{-\delta});\tilde a\,d\tau\Big)$  associated
with the quadratic form
\begin{multline*}
q_{\zeta,\beta,h,\xi}(u)=\\\int_0^{h^{-\delta}}\left(|u'(\tau)|^2+
(1+h^{1/2}\Delta_{\beta,\tau})\Big|\Big(\zeta\tau(1-\frac12\beta h^{1/2}\tau)-\xi\Big)u\Big|^2\right)\big(1-(\beta+mh^\sigma)
h^{1/2}\tau\big)\,d\tau
-|u(0)|^2\,.\end{multline*}
The operator $\Hw
$ is with compact resolvent. The strictly increasing
sequence of the eigenvalues of $\Hw$ is denoted by
$(\lambda_n(\Hw))_{n\in \mathbb N}$.

\subsubsection{Harmonic oscillator on an interval}

Here we study the operator in \eqref{eq:H0b} for $\beta=0$, $m=0$ and $\Delta_{\beta,\tau}=0$ which becomes the harmonic oscillator 
\begin{equation}\label{eq:H0b-b=0}
\mathcal H_{\zeta,0,\xi,h}=-\frac{d^2}{d\tau^2}+(\zeta\tau-\xi)^2\quad{\rm in~}L^2\big((0,h^{-\delta});d\tau)\,,
\end{equation} 
and with the boundary conditions $u'(0)=-u(0)$ and $u(h^{-\delta})=0$.

By comparison of the quadratic forms of the operators $\Hw$ and $\Hwo$, we get that  the spectrum of $\Hw$ is localized near that of
 $\Hw$ as $h$ goes to $0$. This gives us a rough  information  about the spectrum of the operator $\Hw$  precisely stated in: 

\begin{lem}\label{lem:H0b}
Let $0<c_1<c_2$, $0<\epsilon\leq \frac14$ and $\delta\in(0,1)$. 
There exist two constants $C>0$ and $h_0\in(0,1)$ such that, for all
$$h\in(0,h_0),\quad c_1 h^\epsilon\leq \zeta\leq c_2 h^\epsilon\quad{\rm and}~|\beta|+m\leq c_2\,,$$
 it holds the following.
\begin{enumerate}
\item $\lambda_2(\Hw)\geq -C|\beta|h^{\frac12-\delta}$.
\item If $\epsilon=\frac14$, $\delta<\frac14$  and $|\xi|\geq (3c_2+2) h^{\frac14-\delta}$, then
$$\lambda_1(\Hw)\geq -1+h^{\frac12-2\delta}\,.$$
\item Let $A_0$ be the universal constant in Lemma~\ref{lem:Hr}. If $\epsilon<\frac14$, $\delta<\frac12-2\epsilon$ and $|\xi|\geq A_0\zeta$, then
$$\lambda_1(\Hw)\geq -1+\zeta^2\,.$$
\end{enumerate}
\end{lem}
\begin{proof}~

{\bf Step~1.}

There exists a constant $C>0$ such that, for all $u\in
H^1((0,h^{-\delta}))$,
\begin{align*}
&\Big|q_{\zeta,\beta,h,\xi}(u)-q_{\zeta,\beta=0,h,\xi}(u)\Big|\leq C|\beta|h^{\frac12-\delta}\Big(q_{0,h}(u)+\|u\|^2_{L^2((0,h^{-\delta});d\tau)}\Big)\,\\
&\Big|\|u\|^2_{L^2((0,h^{-\delta});(1-\beta h^{1/2}\tau)d\tau)}-\|u\|^2_{L^2((0,h^{-\delta});d\tau)}\Big|\leq |\beta|h^{\frac12-\delta}\,\|u\|^2_{L^2((0,h^{-\delta});d\tau)}\,.
\end{align*}
The min-max principle yields, for all $n\in\mathbb N$,
\begin{equation}\label{eq:Hw=Hw0}
\Big|\lambda_n(\Hw)-\lambda_n(\Hwo)\Big|\leq C|\beta|h^{\frac12-\delta}
\Big(\big|\lambda_n(\Hwo)\big|+1\Big)\,.
\end{equation}
Since the form domain of the operator $\Hr$ contains that of the operator $\Hw$ (cf. \eqref{eq:Hr}), then the min-max principle yields
\begin{equation}\label{eq:Hw0=Hr}
\lambda_n(\Hwo)\geq \lambda_n(\Hr)\,.
\end{equation}
In particular, for $n=2$, Lemma~\ref{lem:Hr*} gives us the statement in the first item of Lemma~\ref{lem:H0b}.

{\bf Step~2.}

We estimate the quadratic form for the operator $\Hwo$ as follows, for all $u\in H^1(0,h^{-\delta})$,
$$q_{\zeta,0,\xi,h}(u)\geq \int_0^{h^{-\delta}}\Big(|u'(t)|^2+(-2c_2h^{\frac14-\delta}\xi+\xi^2)|u|^2\Big)\,d\tau-|u(0)|^2\,.$$
The min-max principle and Lemma~\ref{lem:Hr*} yield,
$$\lambda_1(\Hwo)\geq -1+(|\xi|-2c_2h^{\frac14-\delta})|\xi|\,.$$
We insert this into \eqref{eq:Hw=Hw0}. That way, for $|\xi|\geq (3c_2+2)h^{\frac14-\delta}$, we get the conclusion in the second item of Lemma~\ref{lem:H0b}.

{\bf Step~3.}

Using \eqref{eq:Hw=Hw0} and \eqref{eq:Hw0=Hr} for $n=1$, we get,
$$\lambda_1(\Hw)\geq\lambda_1(\Hwo)-Ch^{\frac12-\delta}\,.$$
Now, we assume that $|\xi|\geq A_0\zeta$, $\epsilon<\frac14$ and $\delta<\frac12-2\epsilon$. By applying  Lemma~\ref{lem:Hr} we get, for $h$ sufficiently small, the statement
in the third item in Lemma~\ref{lem:H0b}.
\end{proof}

\subsubsection{Lower bound for the principal eigenvalue of the operator $\Hw$}

In the next two propositions, we determine  refined lower bounds of the eigenvalue  $\lambda_1(\Hw)$. The bound is valid as $h\to0_+$ and is uniform with respect to the parameters $\xi$, $\zeta$ and $\beta$.

\begin{pro}\label{lem:H0b;l}
Let $0<c_1<c_2$,  $\sigma\in(0,1)$, $\epsilon\in(0,\frac14)$, $\delta\in(0,\frac12-2\epsilon)$ and $n\in\mathbb N$ be the smallest positive integer such that
$$(2n+2)\epsilon>\frac12\,.$$
There exist constants $C>0$ and $h_0\in(0,1)$ such that, for all
$$h\in(0,h_0)\,,\quad c_1h^\epsilon<\zeta\leq c_2h^\epsilon\,,\quad \xi\in\R\,,\quad |\beta|h^\delta<\frac13\,,\quad |\beta|+m\leq c_2\,,$$ it holds,
$$\lambda_1(\Hw)\geq -1+e_n(\zeta)-\beta h^{1/2}-Ch^r\,,$$
where
$$r=\min\big({(2n+2)\epsilon},\frac12+2\epsilon,\frac12+\sigma\big)\,.$$
\end{pro}
\begin{proof}
Let $A_0$ be the universal constant in Lemma~\ref{lem:Hr}. In light of the results in Remark~\ref{rem:en} and Lemma~\ref{lem:H0b},  the lower bound in Lemma~\ref{lem:H0b;l} holds true for $|\xi|\geq A_0\zeta$. It remains to prove the lower bound for $|\xi|\leq A_0\zeta$.

 Consider the function
$$f(\tau)=\chi(\tau\, h^{\delta})\,w_n(\tau)\,,$$
where $w_n(\tau)$ is the function in Theorem~\ref{thm:Hr} and $\chi\in C_c^\infty([0,\infty))$ satisfies
$$0\leq \chi\leq 1~{\rm in~}[0,\infty)\,,\quad \chi=1~{\rm
in~}[0,1/2)\quad{\rm and}\quad\chi=0~{\rm in~}[1/2,\infty)\,.$$
Clearly, the function $f$ is in the domain of the operator $\Hw$. It is easy to check that,
$$\Big|\|f\|^2_{{L^2((0,h^{-\rho});(1-\beta
h^{1/2}\tau)d\tau)}}-1\Big|\leq C\zeta^2\,.$$
In light of Theorem~\ref{thm:Hr} and the expression of  $\Hw f$ in \eqref{eq:H0b}, we may write,
\begin{equation}\label{eq:en-tf-n*}
\begin{aligned}
\|\{\Hw-(-1+f_n(\zeta,\xi)-\beta
h^{1/2})\}f\|_{{L^2((0,h^{-\rho});(1-\beta h^{1/2}\tau)d\tau)}}&\leq
C(\zeta^{2n+2}+h^{\frac12+2\epsilon}+h^{\frac12+\sigma})\\
&\leq Ch^r\,.
\end{aligned}
\end{equation}
By the spectral theorem, we deduce that there exists an eigenvalue
$\lambda(\Hw)$ of $\Hw$ such that
$$\Big|\lambda(\Hw)-(-1+f_n(\zeta,\xi))\Big|\leq
C\,h^r\,.$$
Now, Lemma~\ref{lem:H0b} tells us that
$$\lambda_1(\Hw)=\lambda(\Hw)\,.$$
Finally, by definition of $e_n(\zeta)$ in \eqref{eq:e-n}, we have $f_n(\zeta,\xi)\geq e_n(\zeta)$.
\end{proof}

\begin{pro}\label{lem:H0b;l*}
Let $0<c_1<c_2$, $\sigma\in(0,1)$ and $\delta\in(0,\frac18)$. 
There exist constants $C>0$ and $h_0\in(0,1)$ such that, for all
$$h\in(0,h_0)\,,\quad c_1h^{\frac14}\leq \zeta\leq c_2 h^{\frac14}\,,\quad \xi\in\R\,,\quad |\beta|h^\delta<\frac13\,,\quad |\beta|+m\leq c_2\,,$$ it holds,
$$\lambda_1(\Hw)\geq -1+\frac14\zeta^2-\beta h^{1/2}-Ch^r\,,$$
where
$$r=\min\Big({1-4\delta},\frac12+\sigma\Big)\,.$$
\end{pro}
\begin{proof}
The lower bound in Lemma~\ref{lem:H0b;l*} trivially holds when $|\xi|\geq (3c_2+2)h^{\frac14-\delta}$ thanks to Lemma~\ref{lem:H0b}. 

Now we handle the case where $|\xi|\leq (3c_2+2)h^{\frac14-\delta}$. Let $w_1$ be the function constructed in Theorem~\ref{thm:Hr} and choose $\chi\in C_c^\infty([0,\infty))$ such that
$$0\leq \chi\leq 1~{\rm in~}[0,\infty)\,,\quad \chi=1~{\rm
in~}[0,1/2)\quad{\rm and}\quad\chi=0~{\rm in~}[1/2,\infty)\,.$$   Consider the function
$$f(\tau)=\chi(\tau\, h^{\delta})\,w_n(\tau)\,.$$ 
Clearly, the function $f$ is in the domain of the operator $\Hw$ and
$$\Big|\|f\|^2_{{L^2((0,h^{-\rho});(1-\beta
h^{1/2}\tau)d\tau)}}-1\Big|\leq C\zeta^2\,.$$
Inserting the estimates in Theorem~\ref{thm:Hr} into the expression of $\Hw f$ in \eqref{eq:H0b}, and using that $\zeta=\mathcal O(h^{1/4})$ and $\xi=\mathcal O(h^{\frac14-\delta})$, we may write,
\begin{equation}\label{eq:en-tf-m}
\begin{aligned}
\|\{\Hw-(-1+f_1(\zeta,\xi)-\beta
h^{1/2})\}f\|_{{L^2((0,h^{-\rho});(1-\beta h^{1/2}\tau)d\tau)}}&\leq
C\big(\zeta^{4}+\xi^4+(\zeta^2+\xi^2+h^\sigma)h^{\frac12}+h\big)\\
&\leq Ch^r\,.
\end{aligned}
\end{equation}
Now, the spectral theorem and Lemma~\ref{lem:H0b} yield
$$\lambda_1(\Hw)=-1+f_1(\zeta,\xi)-\beta h^{1/2}+\mathcal O(h^{r})\,.$$
Noticing that $\displaystyle\min_{|\xi|\leq (3c_2+2)h^{\delta-\frac12}}f_1(\zeta,\xi)=\frac14\zeta^2$, we finish the proof of Lemma~\ref{lem:H0b;l*}.
\end{proof}

\section{Analysis of the semi-classical Laplacian with a weak magnetic field}\label{sec:proof}

\subsection{Semi-classical Laplacian with weak magnetic field}

We will introduce a new semi-classical  magnetic Laplacian but with a Robin condition not involving the parameters $\alpha$ and $\gamma$. These two parameters will be absorbed by a new (small) parameter $\zeta$.

For $h>0$ and $\zeta>0$, we introduce the operator
\begin{equation}\label{mag-Lap:P}
P_{h,\zeta}=-(h\nabla-i\zeta\Ab_0)^2\quad{\rm in~}L^2(\Omega)\,,
\end{equation}  
whose domain is
\begin{equation}\label{mag-Lap:dom}
D(P_{h,\zeta})=\{u\in H^1_{h^{-1}\zeta\Ab_0})~:~-(h\nabla-i\zeta\Ab_0)^2\in L^2(\Omega)~{\rm and~}\nu\cdot(h\nabla-i\Ab_0)u=-h^{1/2}u{\rm~on~}\partial\Omega\}\,.
\end{equation}  
This operator is defined via the quadratic form
\begin{equation}\label{mag-Lap:form}
u\mapsto q_{h,\zeta}(u)=\int_\Omega|(h\nabla-i\zeta\Ab_0)u|^2\,dx-h^{3/2}\int_{\partial\Omega}|u|^2\,ds(x)\,.
\end{equation}
Let $\sigma(P_{h,\zeta})$ be the spectrum of the operator $P_{h,\zeta}$. We introduce the ground state energy,
\begin{equation}\label{mag-Lap:gse}
\lambda_1(h,\zeta)=\inf\sigma(P_{h,\zeta})\,.
\end{equation}
There is a relationship between the ground state energies in \eqref{eq:gse} and \eqref{mag-Lap:gse} displayed as follows:
\begin{equation}\label{gse=gse}
\mu_1(h;b,\alpha,\gamma)=\frac{\gamma^4}{h^{2-4\alpha}}\lambda_1\Big(\frac{h^{2-2\alpha}}{\gamma^2},b\frac{h^{1-2\alpha}}{\gamma^2}\Big)\,.
\end{equation}

Now Theorem~\ref{thm:gs-gse} follows from:

\begin{thm}\label{thm:mag-Lap}
Let $0<c_1<c_2$ and $\epsilon>0$. Suppose that
$$0<h<1\quad{\rm and}\quad c_1h^\epsilon\leq\zeta\leq c_2h^{\epsilon}\,.$$ It
holds the following.
\begin{enumerate}
\item
If $\epsilon<\frac14$, then as $h\to0_+$, the ground state energy in \eqref{mag-Lap:gse} satisfies
$$\lambda_1(h,\zeta)=-h+e_n(\zeta)h-\kappa_{\max}h^{3/2}+h^{3/2}o(1)\,,$$
where $e_n(\zeta)$ is introduced in \eqref{eq:e-n} and $n\in\mathbb N$ is the smallest positive integer such that $(2n+2)\epsilon>\frac12$.
\item If $\epsilon=\frac14$, then as $h\to0_+$, the ground state energy in \eqref{mag-Lap:gse} satisfies
$$\lambda_1(h,\zeta)=-h+\frac14\zeta^2h-\kappa_{\max}h^{3/2}+h^{3/2}o(1)\,.$$
\item If $\epsilon>\frac14$, then as $h\to0_+$, the ground state energy in \eqref{mag-Lap:gse} satisfies
$$\lambda_1(h,\zeta)=-h-\kappa_{\max}h^{3/2}+h^{3/2}o(1)\,.$$ 
\item There exist constants $\rho\in(0,\frac12)$, $\eta^*\in(0,\frac14)$, $C>0$ and $h_0\in(0,1)$ such that, for all $h\in(0,h_0)$, every ground state $u_h$ of $\lambda_1(h,\zeta)$ satisfies,
$$\|u_{h,\zeta}\|_{L^2(\Omega_{\rm bnd})}\leq C\exp\left(-\frac12h^{\eta^*-\frac14}\right)\,,\quad
\|u_{h,\zeta}\|_{L^2(\Omega_{\rm int})}\leq \exp\left(-\frac12h^{\rho-\frac12}\right)\,, $$
where
$$
\begin{aligned}
&\Omega_{\rm int}=\{x\in\Omega~:~{\rm dist}(x,\partial\Omega)\geq h^{\rho}\}\quad{\rm and}\\
&\Omega_{\rm bnd}=\{x\in\Omega\setminus\overline\Omega_{\rm int}~:~\kappa_{\max}-\kappa(p(x))\geq h^{\eta^*}\}\,.
\end{aligned}
$$

\end{enumerate}
\end{thm}

The proof of the items  (1)-(3) in Theorem~\ref{thm:mag-Lap} follows from Proposition~\ref{prop:ub} and Proposition~\ref{prop:lb}. We will give explicit bounds to the remainder $h^{3/2}o(1)$ in the form $\mathcal O(h^r)$ where $r$ depends on $\epsilon$ and satisfies $r>\frac32$. More specifically, we find that
$$r=\min(r_*,r^*)\,,$$
where $r_*$ and $r^*$ are introduced in \eqref{eq:ubr*} and \eqref{eq:r*} respectively.

The proof of the item (4) in Theorem~\ref{thm:mag-Lap} follows from Theorems~\ref{thm:dec} and \ref{thm:conc-gs}.

\subsection{Boundary coordinates}\label{sec:app}
We will perform various computations of trial functions supported in a tubular neighborhood of the boundary. To single out the influence of the boundary curvature, we need  a special coordinate system displaying the arc-length along the boundary and the
normal distance to the boundary. We will refer to such coordinates as {\it boundary coordinates}. These are the same coordinates used in the semi-classical analysis of the magnetic Laplcian (cf. \cite{HM, FH-b}).

The boundary coordinates are valid in every connected component of the boundary. For
simplicity, we will  suppose that  $\partial\Omega$ has one   connected component; if
more than one connected component exists, then we use the coordinates in each
connected component independently.
Let
\[
\mathbb{R}/(|\partial \Omega |\mathbb{Z})\ni s\mapsto M(s)\in\partial\Omega
\]
 be the arc-length parametrization of $\partial\Omega$ and oriented  counterclockwise. At the point $M(s)\in\partial\Omega$, $\nu(s)$ is the unit outward normal vector; the  unit tangent vector $T(s)$ and the curvature $\kappa(s)$ are defined as follows
\[
T(s):= M^{\prime}(s)\,,
\]
and
\[
T^{\prime}(s)=\kappa(s)\, \nu(s).
\]
The counterclockwise orientation of the parametrization is displayed as follows,
\[
\forall s\in \mathbb{R}/(|\partial \Omega |\mathbb{Z}) \,,\quad
\det(T(s),\nu(s))=1.
\]
The smoothness of the boundary yields the existence of a constant $t_0>0$ such that, upon defining 
\[
\mathcal{V}_{t_0}= \{x\in \Omega~:~{\rm dist} (x,\partial\Omega)<t_0\},
\]
 the map 
\begin{equation*}
\Phi:\mathbb{R}/(|\partial \Omega |\mathbb{Z})\times(0,t_{0})\ni  (s,t)\mapsto x= M(s)-t\, \nu(s)\in \mathcal{V}_{t_{0}}.
\end{equation*}
becomes a diffeomorphism.
Let us note that, for $x\in\mathcal{V}_{t_{0}}$, one can write
\begin{equation}\label{BC}
\Phi^{-1}(x):=(s(x),t(x))\in \mathbb{R}/(|\partial \Omega |\mathbb{Z})\times (0,t_{0}),
\end{equation}
where $t(x)={\rm dist}(x,\partial\Omega)$ and
$s(x)\in\mathbb{R}/(|\partial \Omega |\mathbb{Z})$ is (uniquely) defined via the reation ${\rm
dist}(x,\partial\Omega)= |x-M(s(x))|$.

Now we express various integrals in the new coordinates $(s,t)$. First, note that the Jacobian determinant  of the transformation $\Phi^{-1}$ is
given by:
\[
a(s,t)=1-t\kappa(s).
\]
In the new coordinates, the components of the vector field $\Ab_0$ are given as follows,
\begin{equation}
\begin{aligned}\label{mf-nc}
\widetilde A_{1}(s,t)&= \Ab_0\cdot\dfrac{\partial x}{\partial s}=(1-t\kappa(s)) \Ab_0(\Phi(s,t))\cdot M^{\prime}(s),\\
\widetilde A_{2}(s,t)&= \Ab_0\cdot\dfrac{\partial x}{ \partial t}= \Ab_0(\Phi(s,t))\cdot\nu(s).
\end{aligned}
\end{equation}
The new magnetic potential $\widetilde\Ab_0=(\widetilde A_1,\widetilde A_2)$ satisfies,
\begin{equation*}
\Big[\dfrac{\partial \widetilde A_{2}}{\partial s}(s,t)- \dfrac{\partial \widetilde A_{1}}{\partial t}(s,t)\Big]ds\wedge dt= \curl\Ab_0(\Phi^{-1}(s,t))dx \wedge  dy= (1-t\kappa(s))ds\wedge dt.
\end{equation*}

For all $u\in L^{2}(\mathcal V_\delta)$, we assign the function  $\widetilde u$ defined in the new coordinates as follows 
\begin{equation}\label{eq:ut}
\widetilde u(s,t):= u(\Phi(s,t)).
\end{equation}
Consequently, for all $u\in H^{1}(\mathcal{V}_{t_{0}})$, we have, with $\widetilde
u=u\circ \Phi$,
\begin{equation}\label{eq:bc;qf}
\int_{\mathcal{V}_{t_{0}}}|(h\nabla-i\zeta\Ab_0) u|^{2}dx= \int \Big
[(1-t\kappa(s))^{-2}|(h\partial_{s}-i\zeta\widetilde A_1) \widetilde u|^{2} +
|(h\partial_{t}-i\zeta\widetilde A_2)\widetilde
u|^{2}\Big](1-t\kappa(s))dsdt\,, \end{equation}
\begin{equation}\label{eq:bc;n}
\int_{\mathcal{V}_{t_{0}}}|u|^{2}dx=\int|\widetilde u(s,t)|^{2}(1-t\kappa(s))dsdt\,,
\end{equation}
and
\begin{equation}\label{eq:bc;tr}
\int_{\mathcal{V}_{t_{0}}\cap\partial\Omega}|u|^{2}dx=\int|\widetilde u(s,t=0)|^{2}\,ds.
\end{equation}
Finally, we recall a useful gauge transformation that we borrow from \cite{HM, FH-b}. Let $x_0\in\partial\Omega$ and $\mathcal V_{x_0}$ be a neighborhood of $x_0$ in $\overline\Omega$. There exists a smooth function $\phi_{x_0}$ in $\Phi^{-1}(\mathcal V_{x_0})$ such that, in the boundary coordinates,
\begin{equation}\label{eq:gaugeA0}
\widetilde A-\nabla_{(s,t)}\phi_{x_0}=\Big(-t+\frac{t^2}2\kappa(s),0\Big)\,.\end{equation}

\subsection{Upper bound for the principal eigenvalue}

In the rest of this paper, we will use the following notation. For all  $\epsilon>0$ and  For all $\zeta\in(0,1)$, define
\begin{equation}\label{eq:b-e-z}
b_\epsilon(\zeta)=\left\{
\begin{array}{ll}
e_n(\zeta)&{\rm if~}\epsilon<1/4\,,\\
\frac14\zeta^2&{\rm if~}\epsilon=1/4\,,\\
0&{\rm if~}\epsilon>1/4\,,
\end{array}\right.
\end{equation}
where $n\in\mathbb N$ is the smallest positive integer satisfying $(2n+2)\epsilon>\frac12$.

\begin{proposition}\label{prop:ub}
Under the assumptions in Theorem~\ref{thm:mag-Lap},  there exist two constants $C>0$ and $h_0\in(0,1)$ such that, for all $h\in(0,h_0)$, the ground state energy in \eqref{mag-Lap:gse} satisfies,
$$\lambda_1(h,\zeta)\leq -h+b_\epsilon(\zeta)h-\kappa_{\max}h^{3/2}+C h^{r_*}\,,$$
where
\begin{equation}\label{eq:ubr*}
r_*=\left\{
\begin{array}{ll}
1+\min\big((2n+2)\epsilon,\frac12+\epsilon\big)&{\rm if~}\epsilon<1/4\,,\\
13/8&{\rm if~}\epsilon\geq1/4\,.\\
\end{array}\right.
\end{equation}
Here $b_\epsilon(\zeta)$  is as in \eqref{eq:b-e-z}.
\end{proposition}
\begin{proof}
The proof consists of  constructing a trial function $v_h(x)$ and computing its energy. This trial function will be defined via the boundary coordinates $(s,t)$ in \eqref{BC}. Select $x_0\in\partial\Omega$ such that 
$$\kappa(s(x_0))=\kappa_{\rm max}$$
is equal to the maximal curvature. We may choose the coordinates $(s,t)$ in \eqref{BC} such that $s(x_0)=0$. Let $\mathcal V_{x_0}$ be a neighborhood of the point $x_0$ in $\overline{\Omega}$, and $\varphi_0=\varphi_{x_0}$ be the function defined in $\mathcal V_{x_0}$ and satisfying \eqref{eq:gaugeA0}.

The construction  of the trial function $v_h$ and the computation of its energy will be done for the cases $\epsilon<\frac14$, $\epsilon=\frac14$ and $\epsilon>\frac14$ independently.  

{\bf  The case $\epsilon<\frac14$.}

Let $n\in\mathbb N$ be the largest positive integer such that $(2n+2)\epsilon>\frac12$. Recall the definition of $f_n(\zeta,\xi)$ in \eqref{eq:f-n}. Select $\xi_n=\xi_n(\zeta)$ such that
$$f_n(\zeta,\xi_n)=e_n(\zeta)=\min\{f_n(\zeta,\xi)~:~|\xi|\leq \zeta\max(A_0,1)\}\quad {\rm and}\quad |\xi_n|\leq \zeta\max(A_0,1)\,.$$

The trial function $v_h$ is defined using the $(s,t)$-coordinates and
the relation in \eqref{eq:ut} as follows,
\begin{equation}\label{eq:ts}
\widetilde v_h(s,t)=c\,h^{-\frac{1+\epsilon}4}\,\chi_1\left(\frac{t}{h^{\rho}}\right)\,\chi_1\left(\frac{s}{h^{\epsilon/2}}\right)\,e^{-i\varphi_0}\,w_n(h^{1/2}\,t)\,\exp\left(\frac{i\xi_n}{h^{1/2}}\right)\,.
\end{equation}
Several objects appear in the definition of $\widetilde v_h$:
\begin{enumerate}
\item $\chi_1\in C_c^\infty(\R)$ satisfies $0\leq \chi\leq 1$ in
$\R$, $\chi_1=1$ in $[-1/2,1/2]$ and ${\rm
supp}\chi_1\subset[-1,1]$\,;
\item $c=\|\chi_1\|^{-1}_{L^2(\R)}$\,;
\item $w_n(\tau)$ in the function in Theorem~\ref{thm:Hr}\,;
\item $\rho=\frac14+\epsilon$\,.
\end{enumerate}
The upper bound in Proposition~\ref{prop:ub} follows from the min-max principle and the following two estimates:
\begin{align}
&\Big|\|v_h\|_{L^2(\Omega)}-1\Big|\leq Ch^{1/2}\,,\label{eq:n-tf}\\
&\Big\|\big(P_{h,\zeta}+h-e_n(\zeta)+\kappa_{\max}h^{3/2}\big)v_h \|_{L^2(\Omega)}\leq h^{r_*}\,,\label{eq:en-tf}
\end{align}
where $r_*=1+\min\Big((2n+2)\epsilon,\frac12+\frac\epsilon2\Big)$ is given in \eqref{eq:ubr*}. 

The estimate in \eqref{eq:n-tf} is easy to obtain in light of the expression of $\widetilde v_h$ and the formula in \eqref{eq:bc;n}. For the estimate in \eqref{eq:en-tf}, notice that, after expressing the operator $P_{\zeta,h}$ in the boundary coordinates $(s,t)$, we get (compare with \eqref{eq:bc;qf})
$$
\Big\|\big(P_{h,\zeta}+h-e_n(\zeta)+\kappa_{\max}h^{3/2}\big)v_h \|_{L^2(\Omega)}=
\Big\|\big(L_{h,\zeta}+h-e_n(\zeta)+\kappa_{\max}h^{3/2}\big)u_h \|_{L^2(a\,dsdt)}\,,
$$ 
where
$$a(s,t)=1-t\kappa(s)\,,$$
$$u_h(s,t)=c\,h^{-\frac{1+\epsilon}4}\,\chi_1\left(\frac{t}{h^{\rho}}\right)\,\chi_1\left(\frac{s}{h^{\epsilon/2}}\right)w_n(h^{1/2}\,t)\,,$$
and
$$
L_{h,\zeta}=-h^2a^{-1}\partial_t(a\partial_t)-a^{-2}\Big(h\partial_s-i\zeta t\big(1-\frac{t}2\kappa(s)\big)\Big)^2\,.
$$
Note that,
\begin{equation}\label{eq:Lhz}
L_{h,\zeta}u_h=ch^{-\frac{1+\epsilon}4}\chi_1\left(\frac{s}{h^{\epsilon/2}}\right)\,PL_{h,\zeta}\chi_1\left(\frac{t}{h^{\rho}}\right)w_n(h^{1/2}\,t)
+R_h\,,
\end{equation}
where
\begin{equation}\label{eq:PL}
PL_{h,\zeta}=-h^2a^{-1}\partial_t(a\partial_t)+a^{-2}\Big(\zeta t\big(1-\frac{t}2\kappa(s)\big)-h^{1/2}\xi_n\Big)^2\,,\end{equation}
and
\begin{multline}\label{eq:Rh}
R_h=ch^{-\frac{1+\epsilon}4}\chi_1\left(\frac{t}{h^{\rho}}\right)w_n(h^{1/2}\,t)
\left[h^2\partial_s^2\chi_1\left(\frac{s}{h^{\epsilon/2}}\right)\right]\\
+2ich^{-\frac{1+\epsilon}4}\chi_1\left(\frac{t}{h^{\rho}}\right)w_n(h^{1/2}\,t)
\left[\Big(h^{1/2}\xi_n-\zeta t\big(1-\frac{t}2\kappa(s)\big)\Big)h\partial_s\Big)\right]\chi_1\left(\frac{s}{h^{\epsilon/2}}\right)\,.\end{multline}
It is easy to check that
\begin{equation}\label{eq:Rh'}
\|R_h\|_{L^2(adsdt)}\leq C\zeta h^{\frac32+\frac\epsilon2}+Ch^{2-\epsilon}\leq Ch^{\frac32+\frac\epsilon2}\,.\end{equation}
Now we perform the change of  variable $t=h^{1/2}\tau$ and get ($\tilde a=1-h^{1/2}\tau\kappa(s)$),
$$
\begin{aligned}
PL_{h,\zeta}&=h\Big[-\tilde a^{-1}\partial_\tau(a\partial_\tau)+\tilde a^{-2}\Big(\zeta \tau\big(1-\frac{h^{1/2}\tau}2\kappa(s)\big)-\xi_n\Big)^2\Big]\\
&=h\mathcal H_{\zeta,\kappa(s),\xi_n,h}\,,
\end{aligned}
$$
where the operator  $\mathcal H_{\zeta,\kappa(s),\xi_n,h}$ is introduced in \eqref{eq:H0b} with $\beta=\kappa(s)$, $\Delta_{\beta,\tau}=h^{-1/2}(\tilde a^{-2}-1)$, $m=0$ and $\delta=\frac12-\rho=\frac12(\frac12-2\epsilon)<\frac12-2\epsilon$. Now, it is easy to prove that (compare with \eqref{eq:en-tf-n*}),
$$\Big\|\Big(h^{-1}PL_{h,\zeta}-\big(-1+f_n(\zeta,\xi_n)-\kappa(s)h^{1/2}\big)\Big)\chi_1\left(h^{\frac12-\rho}\tau\right)w_n(\tau)\Big\|_{L^2(\tilde ad\tau)}\leq Ch^r\,,$$
where
$$r=\min\Big((2n+2)\epsilon,\frac12+2\epsilon\Big)\,.$$
Returning back to the $t$ variable then integrating with respect to the $s$ variable, we get,
$$\Big\|\Big(PL_{h,\zeta}-\big(-h+h^{1/2}f_n(\zeta,\xi_n)-\kappa(s)h^{1/2}\big)\Big)\chi_1\left(h^{-\rho}t\right)w_n(h^{1/2}t)\Big\|_{L^2(adsd\tau)}\leq Ch^{\frac34+r}\,.$$
Now, we insert this and \eqref{eq:Rh'}  into \eqref{eq:Lhz} to get,
$$
\Big\|\big(L_{h,\zeta}+h-e_n(\zeta)+\kappa(s)h^{3/2}\big)u_h \|_{L^2(a\,dsdt)}\leq 
Ch^{\min\Big(r,\frac12+\frac\epsilon2\Big)}=Ch^{r_*}\,,
$$ 
To finish the proof, we notice that $\kappa_{\max}=\kappa(0)$ and in the support of the function $u_h$, we have 
$$|\kappa(s)-\kappa(0)|\leq Ch^{1/8}\,.$$

{\bf The case $\epsilon=\frac14$.}

Now the trial function $v_h$ is defined using the $(s,t)$-coordinates and
the relation in \eqref{eq:ut} as follows,
\begin{equation}\label{eq:ts*}
\widetilde v_h(s,t)=c\,h^{-5/16}\,\chi_1\left(\frac{t}{h^{7/16}}\right)\,\chi_1\left(\frac{s}{h^{1/8}}\right)\,e^{-i\varphi_0}\,w_1(h^{1/2}\,t)\,\exp\left(\frac{i\zeta}{2h^{1/2}}\right)\,,
\end{equation}
where the constant $c$ and the function $\chi_1$ are as in \eqref{eq:ts}, and $w_1$ is the function defined in Theorem~\ref{thm:Hr}.

Performing a calculation similar to the one done for the case $\epsilon<\frac14$ (in particular,
using \eqref{eq:en-tf-m} for $\delta=\frac1{16}$ and $\beta=\kappa(s)$) we get, 
\begin{align*}
&\Big|\|v_h\|_{L^2(\Omega)}-1\Big|\leq Ch^{1/2}\,,\\
&\Big\|\big(P_{h,\zeta}+h-\frac14\zeta^2+\kappa_{\max}h^{3/2}\big)v_h \Big\|_{L^2(\Omega)}\leq Ch^{13/8}\,.
\end{align*}
The min-max principle now yields the desired upper bound for $\lambda_1(h,\zeta)$.

{\bf The case $\epsilon>\frac14$.}

Here we simply take the same trial state for the case without a magnetic field but times a phase (cf. \cite{HeKa, Pan}). Precisely, we define $v_h$ as follows,
 \begin{equation}\label{eq:ts**}
\widetilde v_h(s,t)=c\,h^{-5/16}\,\chi_1\left(\frac{t}{h^{3/8}}\right)\,\chi_1\left(\frac{s}{h^{1/8}}\right)\,e^{-i\varphi_0}\,u_0(h^{1/2}\,t)\,,
\end{equation}
where the constant $c$ and the  function $\chi_1$ are as in \eqref{eq:ts} and $u_0(\tau)=\sqrt{2}\,\exp(-\tau)$. Easy calculations similar to those done for $\epsilon<\frac14$ give us (cf. \cite{HeKa, Pan})
\begin{align*}
&\Big|\|v_h\|_{L^2(\Omega)}-1\Big|\leq Ch^{1/2}\,,\\
&\Big\|\big(P_{h,\zeta}+h+\kappa_{\max}h^{3/2}\big)v_h \|_{L^2(\Omega)}\leq h^{13/8}\,.
\end{align*}
The min-max principle now yields the upper bound for $\lambda_1(h,\zeta)$.
\end{proof}

\subsection{Concentration of bound states near the boundary}

\begin{thm}\label{thm:dec} 
Let $0<c_1<c_2$, $\epsilon>0$ and  $\alpha <
1$. There exist constants $C>0$  and $h_0\in(0,1)$ such
that, if $h\in(0,h_0)$, $\zeta\in(c_1h^{\epsilon},c_2h^\epsilon)$ and $u_{h,\zeta}$ is a $L^2$-normalized ground state  of $P_{h,\zeta}$, then,
$$\int_\Omega \left(|u_{h,\zeta}(x)|^2+h^{-1}|(h\nabla-i\zeta\Ab_0) u_{h,\zeta}(x)|^2\right)\exp\left(\frac{2\alpha\, {\rm
dist}(x,\partial\Omega)}{h^{1/2}}\right)\,dx\leq C\,.$$
\end{thm}

The proof of Theorem~\ref{thm:dec} makes use of the result in:

\begin{lem}\label{lem:dec*}
Under the assumptions in Theorem~\ref{thm:dec}, if $0<\rho\leq \frac12$, $w\in H^1(\Omega)$ and ${\rm supp}\,w\subset \{{\rm dist}(x,\partial\Omega)\leq 2h^\rho\}$, then
$$\int_\Omega |(h\nabla-i\zeta\Ab_0)w|^2\,dx-h^{3/2}\int_{\partial\Omega}|w|^2\,dx\geq
-\frac{h}2\int_\Omega|w|^2\,dx\,.$$
\end{lem}
\begin{proof}
Let $q_{h,\zeta}(\cdot)$ be the quadratic form in \eqref{mag-Lap:form}. The diamagnetic inequality yields,
$$
q_{h,\zeta}(w)\geq
\int_\Omega\big|\,h\nabla|w|\,\big|^2\,dx-h^{3/2}\int_{\partial\Omega}|w|^2\,ds(x)\,.
$$
In boundary coordinates, the inequality reads (cf. \eqref{eq:bc;qf}),
$$
q_{h,\zeta}(w)\geq
(1-Ch^{\rho})\iint\big|\,h\nabla v\,\big|^2\,dsdt-h^{3/2}\int|v(s,t=0)|^2\,ds\,,
$$
where $v=|w\circ \Phi(s,t)|$ (cf. \eqref{eq:ut}). Applying the change of the variable $t=h^{1/2}\tau$ and comparing with the operator in \eqref{defH00}, we get,
$$
q_{h,\zeta}(w)\geq
-(1-Ch^{\rho})h\lambda_1(\mathcal H_{0,0})\iint|v(s,t)|dsdt=-(1-Ch^{\rho})h\iint|v(s,t)|dsdt\,.
$$
Returning back to Cartesian coordinates, we get the inequality in Lemma~\ref{lem:dec*}.
\end{proof}

\begin{proof}[Proof of Theorem~\ref{thm:dec}]
The proof is similar to that of Theorem~5.1 in \cite{HeKa}. Let $t(x)={\rm dist }(x,\partial\Omega)$ and
$\Phi(x)=\exp(\frac{\alpha\, t(x)}{h^{1/2}})$. We perform an integration by parts to write the following identity,
\begin{equation}\label{eq:decomp}
\begin{aligned}
q_h^\Phi(u_{h,\zeta})&:=\int_\Omega\left(|(h\nabla-i\zeta\Ab_0)(\Phi\,u_{h,\zeta})|^2-h^2|\nabla
\Phi|^2|u_{h,\zeta}|^2\right)\,dx-h^{3/2}\int_{\partial\Omega}|\Phi\,u_{h,\zeta}|^2\,ds(x)\\
&=\lambda_1(h,\zeta)\|\Phi\,u_{h,\zeta}\|^2_{L^2(\Omega)}\,.
\end{aligned}
\end{equation}
Consider a partition of unity of $\R$
$$\chi_1^2+\chi_2^2=1\,,$$
{such that $\chi_1=1$ in $(-\infty,1)$, ${\rm supp }\chi_1\subset
(-\infty,2)$, $\chi_1\geq 0$ and $\chi_2\geq 0$ in $\R$. }

Define
$$\chi_{j,h}(x)=\chi_j\left(\frac{t(x)}{h^{1/2}}\right)\,,\quad
j\in\{1,2\}\,.$$ Associated with this partition of unity,  we have
the simple standard decomposition
\begin{equation}\label{eq:decomp*}
q_{h,\zeta}^\Phi(u_{h,\zeta})=\sum_{j=1}^2 q_{j,h,\zeta}^\Phi(\,u_{h,\zeta})\,,\end{equation}
where (by Lemma~\ref{lem:dec*})
\begin{equation}\label{eq:decompa}
\begin{aligned}
q_{1,h,\zeta}^\Phi(u_{h,\zeta})&=\int_\Omega\left(|(h\nabla-i\zeta\Ab_0)(\chi_{1,h}\,\Phi\,u_{h,\zeta})|^2-h^2|\nabla
(\chi_{1,h}\Phi)|^2|u_{h,\zeta}|^2\right)\,dx\\
&\qquad-h^{3/2}\int_{\partial\Omega}|\chi_{1,h}\Phi\,u_{h,\zeta}|^2\,ds(x)\\
&\geq -\frac{h}2\int_\Omega|\chi_{1,h}\,\Phi\,u_{h,\zeta}|^2\,dx-Ch\int_\Omega|\Phi\,u_{h,\zeta}|^2\,dx\geq -Ch \,,
\end{aligned}
\end{equation}
and 
\begin{equation}\label{eq:decompb}
 q_{2,h,\zeta}^\Phi(u_{h,\zeta})=\int_\Omega  \left( |(h\nabla-i\zeta\Ab_0)(\chi_{2,h}\,\Phi\,u_{h,\zeta})|^2 - h^2|\nabla
(\chi_{2,h}\Phi)|^2|u_{h,\zeta}|^2\right) \,dx\,.
\end{equation}
The definition of $\Phi$ and the fact $|\nabla t(x)|=1$ a.e. together yield
$$\int_\Omega|\nabla
(\chi_{2,h}\Phi)|^2|u_{h,\zeta}|^2\,dx\leq \alpha^2h\int_\Omega |\chi_{2,h}\Phi u_{h,\zeta}|^2\,dx+Ch\int_\Omega|u_{h,\zeta}|^2\,dx\,.$$
We insert this and \eqref{eq:decompa} into \eqref{eq:decomp}, write $\lambda_1(h,\zeta)\leq \frac12(-1-\alpha^2)$ by Proposition~\ref{prop:ub}  and rearrange the terms to obtain,  
\begin{equation}\label{eq:decomp1bis}
\int_\Omega\left(|(h\nabla-i\zeta\Ab_0)(\chi_{2,h}\,\Phi\,u_{h,\zeta})|^2+\frac12(1-\alpha^2
) h |\chi_{2,h}\,\Phi\,u_{h,\zeta}|^2  \right)\,dx\leq  C
h\,.\end{equation} 
This is enough to deduce the estimate in Theorem~\ref{thm:dec}.
\end{proof}

We record the following simple corollary of Theorem~\ref{thm:dec}. 

\begin{corollary}\label{corol:dec}
Let $\rho\in(0,\frac12)$, $\epsilon>0$ and $0<c_1<c_2$. There exists $h_0\in(0,1)$ such that, for all $h\in (0,h_0)$ and $\zeta\in(c_1h^\epsilon,c_2h^\epsilon)$, every $L^2$-normalized ground state $u_{h,\zeta}$ of the operator $P_{h,\zeta}$ satisfies,
$$\int_{c_1h^{\rho}\leq {\rm dist}(x,\partial\Omega)\leq c_2h^{\rho}}|u_{h,\zeta}|^2\,dx\leq\exp\left(-c_1h^{\rho-\frac12}\right)\,.$$ 
\end{corollary}

\subsection{Lower bound for the principal eigenvalue}

\begin{proposition}\label{prop:lb}
Let $\epsilon>0$ and $0<c_1<c_2$.  There exist constants $C>0$, $h_0\in(0,1)$ and $r^*>\frac32$ such that, for all $h\in(0,h_0)$ and $\zeta\in(c_1h^\epsilon,c_2h^\epsilon)$, the ground state energy in \eqref{mag-Lap:gse} satisfies,
$$\lambda_1(h,\zeta)\geq -h+b_\epsilon(\zeta)h-\kappa_{\max}h^{3/2}-C h^{r^*}\,,$$
where $b_\epsilon(\zeta)$  is  introduced in \eqref{eq:b-e-z}.
\end{proposition}

For $\epsilon>\frac14$, Proposition~\ref{prop:lb} follows  from:
\begin{lem}\label{lem:lb***}
Suppose that $\epsilon>\frac14$. Under the assumption in Proposition~\ref{prop:lb}, for all $u$ in the form domain of the operator $P_{h,\zeta}$,
$$q_{h,\zeta}(u)\geq \int_\Omega U_{h,\zeta}(x)|u|^2\,dx\,,$$
where
$$U_{h,\zeta}(x)=\left\{\begin{array}{ll}
-h-\kappa(s(x))h^{3/2}-Ch^{7/4}&{\rm if~}{\rm dist}(x,\partial\Omega)< 2h^{1/8}\,,\\
0&{\rm if~}{\rm dist}(x,\partial\Omega)\geq 2h^{1/8}\,,
\end{array}\right.
$$
and $q_{h,\zeta}(\cdot)$ is the quadratic form in \eqref{mag-Lap:form}. 
\end{lem}
\begin{proof}
This is a consequence  of the diamagnetic inequality and \cite[Thm.~5.2]{HeKa}.
\end{proof}

In the case $\epsilon\leq \frac14$, Proposition~\ref{prop:lb} is a consequence of Lemma~\ref{lem:lb} below (applied with $w=u_{h,\zeta}$ and $u_{h,\zeta}$ a $L^2$ normalized ground state of the operator $P_{h,\zeta}$) and the variational min-max principle.

The constant $r^*$ in Proposition~\ref{prop:lb}  depends on $\epsilon$. It is introduced as follows. For all $\epsilon>0$, let 
\begin{equation}\label{eq:sigma}
\sigma=
\left\{
\begin{array}{ll}
\frac15\min\big(2\epsilon,1-4\epsilon\big)&{\rm if}~\epsilon<\frac14\,,\\
1/8&{\rm if}~\epsilon\geq\frac14\,,\\
\end{array}
\right.
\end{equation}
\begin{equation}\label{eq:rho}
\rho=
\left\{
\begin{array}{ll}
\frac12-\frac{1}4\min\big(2\epsilon,1-4\epsilon\big)&{\rm if~}\epsilon<1/4\\
7/16&{\rm if~}\epsilon=1/4\,,\\
1/8&{\rm if~}\epsilon>1/4\,,
\end{array}
\right.
\end{equation}
\begin{equation}\label{eq:r*}
r^*=\left\{\begin{array}{ll}
\min\Big(1+(2n+2)\epsilon,\frac32+2\epsilon,\frac32+\sigma,2\epsilon+4\rho+2\sigma-\frac12,1+\sigma+\rho,2-2\sigma\Big)&{\rm if~}\epsilon\leq\frac14\,,\\
\frac74&{\rm if~}\epsilon>\frac14\,,
\end{array}\right.\end{equation}
where $n\in\mathbb N$ is the smallest positive integer satisfying $(2n+2)\epsilon>\frac12$.

Note that $0<\rho<\frac12$, $0<\sigma<1$ and when $\epsilon\leq \frac14$, the following three conditions are satisfied
\begin{equation}\label{eq:3cond}
\left\{
\begin{array}{l}
2\epsilon+4\rho+2\sigma-\frac12>\frac32\,,\\
2-2\sigma>\frac32\,,\\
\rho+\sigma>\frac12\,.
\end{array}\right.\end{equation}
Consequently, for all $\epsilon>0$, the number $r^*$ satisfies 
$$r^*>\frac32\,.$$

\begin{lem}\label{lem:lb}
Let $M>0$, $0<\epsilon\leq \frac14$ and $0<c_1<c_2$.
There exist  two constants $C>0$ and $h_0\in(0,1)$ such that, for all $h\in(0,h_0)$ and 
$\zeta\in(c_1h^\epsilon,c_2h^\epsilon)$, if $w$ is a $L^2$ normalized function in the form domain of  the operator $P_{h,\zeta}$ and
\begin{equation}\label{eq:cond-dec}
\left\|\exp\left(\frac{{\rm dist}(x,\partial\Omega)}{2h^{\rho'}}\right)w\right\|_{L^2(\Omega)}\leq M\,,
\end{equation}
for some $\rho<\rho'<\frac12$, 
then it holds the following:
\begin{equation}\label{eq:bnd-lb1*}
q_{h,\zeta}(w)\geq \int_\Omega U_{h,\zeta}(x) |w|^2\,dx-Ch^2.\end{equation}
Here
\begin{itemize}
\item $U_{h,\zeta}(x)=\left\{\begin{array}{ll}
-h+b_\epsilon(\zeta)-\kappa(s(x))h^{3/2}-Ch^{r^*}&{\rm if~}{\rm dist}(x,\partial\Omega)<2h^{\rho}\,,\\
0&{\rm if~}{\rm dist}(x,\partial\Omega)\geq 2h^{\rho}\,;
\end{array}\right.
$
\item $\sigma$, $\rho$ and $r^*$ are introduced in \eqref{eq:sigma}, \eqref{eq:rho} and \eqref{eq:r*} respectively\,;
\item   $b_\epsilon(\zeta)$ is introduced in \eqref{eq:b-e-z}\,;
\item   $q_{h,\zeta}(\cdot)$ is the quadratic form introduced in \eqref{mag-Lap:form}.
\end{itemize}
\end{lem}

\begin{proof}[Proof of Lemma~\ref{lem:lb}]~

The lengthy proof of Lemma~\ref{lem:lb} is divided into four steps. 

{\bf Step~1. Localization near the boundary.}

Consider a partition of unity of $\R$,
$$
\chi_1^2+\chi_2^2=1
$$
with $\chi_1=1$ in $(-\infty,1]$, ${\rm supp}\chi_1\subset(-\infty,2]$ and ${\rm
supp}\chi_2\subset[1,\infty)$.
For $j\in\{1,2\}$, put,
$$
\chi_{j,h}(x)=\chi_j\left(\frac{{\rm
dist}(x,\partial\Omega)}{h^{\rho}}\right)\,.
$$
We have the  decomposition
$$q_{h,\zeta}(w)=q_{h,\zeta}(\chi_{1,h}w)+q_{h,\zeta}(\chi_{2,h}w)-h^2\sum_{j=1}^2\big\|\,|\nabla\chi_{j,h}|w\,\big\|^2_{L^2(\Omega)}\,,$$
where
$$
q_{h,\zeta}(\chi_{2,h}w)=\int_\Omega|h\nabla
(\chi_{2,h}\,w)|^2\,dx\geq0\,,
$$
and by \eqref{eq:cond-dec},  
$${h^2} \big\|\,|\nabla\chi_{j,h}|w\,\big\|^2_{L^2(\Omega)}\leq
h^{2-2\rho}\exp\left(-\frac14h^{\rho-\rho'}\right)=\mathcal O(h^\infty)\,.$$
Thus,
\begin{equation}\label{eq:=bnd}
q_{h,\zeta}(w)\geq q_{h,\zeta}(\chi_{1,h}\,w)+\mathcal O(h^\infty)\,.
\end{equation}

{\bf Step~2. Analysis near the boundary.}

Let us cover the boundary $\partial\Omega$ by a family of open disks $(B(x_j,h^{1/8}))$. Let $(f_j)\subset C^\infty(\partial\Omega)$ be a partition of unity in $\partial\Omega$ such that, for all $j$,
$${\rm supp}\, f_j\subset B(x_j,h^{\sigma})\cap\partial\Omega\,,\quad{\rm and}\quad |\nabla f_j|\leq Ch^{-\sigma}\,.$$
We extend $f_j$ in the tubular neighborhood $\{{\rm dist}(x,\partial\Omega)<2h^{\rho}\}$ of the boundary via the formula
$$f_j(x)=f_j(s(x))\,.$$
 We decompose the boundary term in \eqref{eq:=bnd} as follows,
\begin{equation}\label{eq:bnd-t}
\begin{aligned}
q_{h,\zeta}(\chi_{1,h}\,w)&=\sum_{j}q_{h,\zeta}(f_{j}\chi_{1,h}w)-h^2\sum_j\|\,|\nabla f_j|\chi_{1,h}w\|_{L^2(\Omega)}^2\\
&\geq \sum_{j}q_{h,\zeta}(f_{j}\chi_{1,h}w)-Ch^{2-2\sigma}\|\chi_{1,h}w\|_{L^2(\Omega)}^2\,.
\end{aligned}
\end{equation}
We will write a lower bound for each term  $q_{h,\zeta}(f_{j}\chi_{1,h}w)$ as follows. First, let us denote by
$$\kappa_j=\kappa(s(x_j))\,.$$
By smoothness of the scalar curvature and boundedness of the boundary, we know that
$$|\kappa(s(x))-\kappa_j|\leq m h^{\sigma}\quad{\rm in~}B(x_j,h^{\sigma})\,,$$
where 
$$m=\sup_{x\in\partial\Omega}|\kappa'(s(x))|\,.$$
That way we get the following pointwise lower bound in every $B(x_j,h^{\sigma})$,
$$
|(h\partial_s-i\zeta t(1-\frac12t\kappa(s))\widetilde w|^2\geq (1-h^{1/2})|(h\partial_s-i\zeta t(1-\frac12t\kappa_j)\widetilde w|^2-4\zeta^2t^4h^{2\sigma-\frac12}|\widetilde w|^2\,.
$$
Let $\phi_j=\phi_{x_j}$ be the function satisfying  \eqref{eq:gaugeA0} in $B(x_j,2h^\sigma+2h^\rho)$.
Define the function\break$v_j=\tilde f_j\tilde \chi_{1,h}\widetilde w\,e^{-i\varphi_j}$.
We express the quadratic form $q_{h,\zeta}(f_j\chi_{1,h}w)$ in boundary coordinates  and  then we use the aforementioned inequalities to write,
\begin{equation}\label{eq:bnd-t*}
\begin{aligned}
&q_{h,\zeta}(f_j\chi_{1,h}w)\geq\\
& \iint\Big(|h\partial_t v_j|^2+(1-h^{1/2})(1-t\kappa_j-Ch^\sigma t)^{-2}
|(h\partial_s-i\zeta t(1-\frac12t\kappa_j)\tilde w|^2 \Big)(1-\kappa_j t-mh^\sigma t)dsdt\\
&-\int|v_j(s,t=0)|^2ds-C\iint\zeta^2t^4h^{2\sigma-\frac12}|v_j|^2(1-\kappa_j t-mh^\sigma t)dsdt\,.
\end{aligned}
\end{equation}

Let 
\begin{equation}\label{eq:lb-parameters}
\delta=\frac12-\rho,\quad\beta=\kappa_j,\quad \Delta_{\beta,\tau}=h^{-1/2}\Big[(1-h^{1/2})\big(1-(\kappa_j+Ch^\sigma)h^{1/2}\tau\big)^{-2}-1\Big]\,.
\end{equation} 
Note that, for $t\in(0,h^{\frac12-\rho})$ and $\tau=h^{-\frac12}t$, 
$|\Delta_{\beta,\tau}|\leq C(|\beta|+1)\tau$. Thus, we can apply the results in Sec.~\ref{sec:Hw}.

We return back to \eqref{eq:bnd-t*}. Note that $\zeta=\mathcal O(h^\epsilon)$ and in the support of $v_j$, the term $t^4$ is of order $\mathcal O(h^{4\rho})$. We apply the change of variable $t=h^{1/2}\tau$ then the Fourier transform with respect to the variable $s$ to obtain,
\begin{equation}\label{eq:bnd-t**}
q_{h,\zeta}(f_j\chi_{1,h}w)\geq \iint \left\{\Big(h\inf_{\xi\in\R}\lambda_1(\Hw)\Big)-Ch^{2\epsilon+4\rho+2\sigma-\frac12}\right\} |f_j\chi_{1,h}w|^2(1-\kappa_jt-mh^\sigma t)\,dsdt\,.
\end{equation}

{\bf Step~3. Lower bound in the case $\epsilon<\frac14$.}

By the assumption on $\rho$ and $\sigma$, we find that
$$0<\delta=\frac12-\rho<\frac12-2\epsilon$$
so that we can apply Proposition~\ref{lem:H0b;l}. Let $r^*>\frac32$ be the constant introduced in \eqref{eq:r*}
 We infer from \eqref{eq:bnd-t**},
\begin{equation}\label{eq:bnd-t**1}
q_{h,\zeta}(f_j\chi_{1,h}w)\geq \iint \left\{-h+he_n(\zeta)-\kappa_jh^{1/2}-Ch^{r_*}\right\} |f_j\chi_{1,h}w|^2\,(1-t\kappa_j-mh^\sigma t)dsdt\,.
\end{equation}
Now, in \eqref{eq:bnd-t**}, we replace $\kappa_j$ by $\kappa(s)+\mathcal O(h^\sigma)$ and use that $\sigma+\rho>\frac12$ to replace the term $1-t\kappa_j-mh^\sigma t$ by $1-t\kappa(s)$  and get,
\begin{equation}\label{eq:bnd-t**1*}
q_{h,\zeta}(f_j\chi_{1,h}w)\geq \iint \left\{-h+he_n(\zeta)-\kappa(s)h^{1/2}-Ch^{r_*}\right\} |f_j\chi_{1,h}w|^2(1-t\kappa(s))\,dsdt\,.
\end{equation}
We insert \eqref{eq:bnd-t**1*} into \eqref{eq:bnd-t} and obtain
$$
\begin{aligned}
q_{h,\zeta}(\chi_{1,h}w)&\geq \sum_j\iint \left\{-h+he_n(\zeta)-\kappa(s)h^{1/2}-Ch^{r_*}\right\} |f_j\chi_{1,h}w|^2(1-t\kappa(s))\,dsdt-Ch^{2}\\
&=\int_\Omega\left\{-h+he_n(\zeta)-\kappa(s(x))h^{1/2}-Ch^{r_*}\right\}|\chi_{1,h}w|^2\,dx-Ch^2\,.
\end{aligned}
$$
Using that $|\chi_{1,h}w|\leq |w|$ and that $-h+he_n(\zeta)-\kappa(s)h^{1/2}-Ch^{r_*}<0$, we get further,
 $$
q_{h,\zeta}(\chi_{1,h}w)
=\int_\Omega\left\{-h+he_n(\zeta)-\kappa(s)h^{1/2}-Ch^{r_*}\right\}|w|^2\,dx-Ch^2\,.
$$
Finally, we insert this into \eqref{eq:=bnd} to get \eqref{eq:bnd-lb1*}.

{\bf Step~4. Lower bound in the case $\epsilon=\frac14$.}

The analysis here is similar to that in Step~3 and we will be rather succinct. Note that the assumption on $\delta$ and $\rho$ ensure that $0<\delta=\frac12-\rho=\frac1{16}<\frac18$ so that we can apply Proposition~\ref{lem:H0b;l*} and infer from \eqref{eq:bnd-t**},
\begin{equation}\label{eq:bnd-t***}
\begin{aligned}
q_{h,\zeta}(f_j\chi_{1,h}w)&
\geq \iint \left\{-h+\frac14\zeta^2h-\kappa_jh^{3/2}-Ch^{r^*}\right\} |f_j\chi_{1,h}w|^2(1-\kappa_jt-mh^\sigma t)\,dsdt\\
&\geq \iint \left\{-h+\frac14\zeta^2h-\kappa(s)h^{3/2}-Ch^{r^*}\right\} |f_j\chi_{1,h}w|^2(1-\kappa(s) t)\,dsdt\,.
\end{aligned}
\end{equation}
We insert this into \eqref{eq:bnd-t} and \eqref{eq:=bnd} to get \eqref{eq:bnd-lb1*} for $\epsilon=\frac14$.
\end{proof}

\subsection{Concentration of ground states near the points of maximal curvature}

\begin{thm}\label{thm:conc-gs}
Let $0<c_1<c_2$ and $\epsilon>0$. There exist constants $\rho\in(0,\frac12)$, $\eta^*\in(0,\frac14)$, $C>0$ and $h_0\in(0,1)$ such that, for all $h\in(0,h_0)$, $\zeta\in(c_1h^\epsilon,c_2h^\epsilon)$ and
$u_{h,\zeta}$ a normalized ground of the operator $P_{h,\zeta}$,
$$\int_{\{{\rm dist}(x,\partial\Omega)\leq h^\varsigma\}\cap\{\kappa_{\max}-\kappa(s(x))\geq h^{\eta^*}\}}|u_{h,\zeta}|^2\,dx
\leq C\exp\left(-h^{\eta^*-\frac14}\right)\,.$$  
\end{thm}

We will prove Theorem~\ref{thm:conc-gs} in the case $0<\epsilon\leq \frac14$. The case $\epsilon>\frac14$ is a standard consequence of the inequality in Lemma~\ref{lem:lb***}
(cf. \cite[Thm.~8.3.4]{FH-b}).

An important ingredient in the proof of Theorem~\ref{thm:conc-gs} is:

\begin{lem}\label{lem:conc-gs}
Let $\epsilon>0$ and $\rho$ be as in \eqref{eq:rho}. Let $\rho'\in(\rho,\frac12)$. There exist two constants $\bar C>0$ and $h_0\in(0,1)$ such that, for all $h\in(0,h_0)$ and $u$ in the form domain of the operator $P_{h,\zeta}$, 
$$q_{h,\zeta}(u)\geq \int_\Omega V_{h,\zeta}(x)|u|^2\,dx\,,$$
where
$$V_{h,\zeta}(x)=
\left\{
\begin{array}{ll}
-h/2&{\rm if~}{\rm dist}(x,\partial\Omega)\geq 2h^{\rho}\,,\\
-h+b_\epsilon(\zeta)h-\kappa(s(x))h^{3/2}-\bar Ch^{r^*}&{\rm if~}{\rm dist}(x,\partial\Omega)< 2h^{\rho}\,,
\end{array}
\right.
$$
and $q_{h,\zeta}$ is the quadratic form in \eqref{mag-Lap:form}.
\end{lem}
\begin{proof}
Let $\tilde\mu$ be the ground state energy of the operator $P_{h,\zeta}-V_{h,\zeta}$. We will prove that $\tilde\mu>0$.

 The min-max principle and Theorem~\ref{thm:dec} together yield
$$\tilde\mu\leq \langle (P_{h,\zeta} -V_{h,\zeta})u_{h,\zeta},
u_{h,\zeta}\rangle_{L^2(\Omega)}=\lambda_1(h,\zeta)-\int_{\Omega}V_{h,\zeta}|u_{h,\zeta}|^2\,dx\leq \tilde{C} h^{r^*}\,.$$

Let $w$ be a $L^2$ normalized ground state of the operator $P_{h,\zeta}-V_{h,\zeta}$. We will prove that,
\begin{equation}\label{eq:dec-new}
\left\|\exp\left(\frac{ t(x)}{h^{\rho'}}\right)w\right\|_{L^2(\Omega)}^2\leq C\,,
\end{equation}
where 
 $t(x)={\rm dist }(x,\partial\Omega)$ and
$\Phi(x)=\exp(\frac{t(x)}{h^{\rho'}})$. We perform an integration by parts to write the following identity,
\begin{equation}\label{eq:decomp**}
\begin{aligned}
\tilde q_h^\Phi(u_{h,\zeta})&:=\int_\Omega\left(|(h\nabla-i\zeta\Ab_0)(\Phi w)|^2-V_h|\Phi w|^2-h^2|\nabla
\Phi|^2|u_{h,\zeta}|^2\right)\,dx-h^{3/2}\int_{\partial\Omega}|\Phi\,w|^2\,ds(x)\\
&=\tilde \mu\|\Phi\,w\|^2_{L^2(\Omega)}\,.
\end{aligned}
\end{equation}
Consider a partition of unity of $\R$
$$\chi_1^2+\chi_2^2=1\,,$$
{such that $\chi_1=1$ in $(-\infty,1)$, ${\rm supp }\chi_1\subset
(-\infty,2)$, $\chi_1\geq 0$ and $\chi_2\geq 0$ in $\R$. }

Define
$$\chi_{j,h}(x)=\chi_j\left(\frac{t(x)}{h^{\rho'}}\right)\,,\quad
j\in\{1,2\}\,.$$ Associated with this partition of unity,  we have
the simple standard decomposition
\begin{equation}\label{eq:decomp*}
\tilde q_{h,\zeta}^\Phi(w)=\sum_{j=1}^2 \tilde q_{j,h,\zeta}^\Phi(w)\,,\end{equation}
where
\begin{multline}\label{eq:decompa*}
\tilde q_{1,h,\zeta}^\Phi(w)=\int_\Omega\left(|(h\nabla-i\zeta\Ab_0)(\chi_{1,h}\,\Phi\,w)|^2-V_{h,\zeta}|\chi_{1,h}\Phi w|^2-h^2|\nabla
(\chi_{1,h}\Phi)|^2|w|^2\right)\,dx\\-h^{3/2}\int_{\partial\Omega}|\chi_{1,h}\Phi\,w|^2\,ds(x)\,,
\end{multline}
and
\begin{equation}\label{eq:decompb*}
\tilde q_{2,h,\zeta}^\Phi(w)=\int_\Omega  \left( |(h\nabla-i\zeta\Ab_0)(\chi_{2,h}\,\Phi\,w)|^2-V_{h,\zeta}|\chi_{2,h}\Phi w|^2 - h^2|\nabla
(\chi_{2,h}\Phi)|^2|w|^2\right) \,dx\,.
\end{equation}
Lemma~\ref{lem:dec*}, the bound $V_{h,\zeta}\leq 0$ and the normalization of $u_{h,\zeta}$  together yield
$$\tilde q^\Phi_{1,h,\zeta}(u_{h,\zeta})\geq-Ch\,.$$
We insert this into \eqref{eq:decomp*}, then we insert the obtained inequality into \eqref{eq:decomp}, use the  bounds $-V_{h,\zeta}\geq \frac12h$, $\tilde\mu\leq \tilde Ch^{r^*}$  and then rearrange the terms to obtain,  
\begin{multline}\label{eq:decomp1bis*}
\int_\Omega\left(|(h\nabla+i\zeta\Ab_0)(\chi_{2,h}\,\Phi\,u_{h,\zeta})|^2+h(\frac12-
 Ch^{1-2\rho'}-\tilde Ch^{r^*-1}) |\chi_{2,h}\,\Phi\,u_{h,\zeta}|^2  \right)\,dx\,\leq   Ch\,.
\end{multline} 
Using the inequality $\frac12-Ch^{1-2\rho'}-Ch^{r^*-1}\geq\frac14$ then dividing by $h$,  we get
$$\int_\Omega\left(h^{-1}|(h\nabla-i\zeta\Ab_0)(\chi_{2,h}\,\Phi\,w)|^2 +\frac14
|\chi_{2,h}\,\Phi\,w|^2\right)\,dx\leq C\,.$$ This is enough to deduce the estimate in \eqref{eq:dec-new}.

Having proved \eqref{eq:dec-new}, we may use the result in Lemma~\ref{lem:lb} and write,
$$q_{h,\zeta}(w)\geq \int_{\Omega}U_{h,\zeta}(x)|w|^2\,dx-Ch^2\,.$$
Note that, by selecting $\bar C>0$ sufficiently large, we may write
$$U_{h,\zeta}(x)-Ch^2\geq V_{h,\zeta}(x)\quad{\rm in~}\Omega\,.$$
This finishes the proof of Lemma~\ref{lem:conc-gs}.
\end{proof}

\begin{proof}[Proof of Theorem~\ref{thm:conc-gs}]
Define the function $\phi(s)=\kappa_{\max}-\kappa(s)$. The function $\phi$ defines a $C^1$ function in $\partial\Omega$. The boundedness of $\partial\Omega$ ensures the existence of a constants $C_0>0$ such that
$$\forall~s\in(0,|\partial\Omega|),\quad |\phi'(s)|^2\leq C\phi(s)\,.$$
Let $\sigma$ and $\rho$ be as in \eqref{eq:sigma} and \eqref{eq:rho}. Choose $\chi\in C_c^\infty(\R)$ such that $0\leq \chi\leq 1$ in $\R$, $\chi=1$ in $[t_0,t_0]$ and ${\rm supp}\,\chi\subset[-1,1]$. Here $t_0<1$ is a geometric constant such that the boundary coordinates $(s,t)$ are valid in the tubular neighborhood $\{{\rm dist}(x,\partial\Omega)<t_0\}$ (cf. Sec.~\ref{sec:app}).

Define the function
$$\Psi(x)=\exp\left(\frac{\delta\chi(t(x))\phi(s(x))}{h^{1/4}}\right)\,,$$
where $t(x)={\rm dist}(x,\partial\Omega)$ and $\delta\in(0,1)$. We will fix a choice for $\delta$ at a later point.

Let us write the following decomposition formula
$$\lambda_1(h,\zeta)\left\|\Psi u_{h,\zeta}\right\|^2_{L^2(\Omega)}+h^{3/2}\delta^2\left\||\nabla\psi|\Psi u_{h,\zeta}\right\|^2_{L^2(\Omega)}=q_{h,\zeta}\left(\Psi u_{h,\zeta}\right)\,,$$
where
$$\psi(x)=\chi(t(x))\phi(s(x))\,.$$
Using Theorem~\ref{thm:dec} and the bound $|\phi'|^2\leq C\phi$, we may write,
$$\left\||\nabla\psi|\Psi u_{h,\zeta}\right\|^2_{L^2(\Omega)}\leq 
\mathcal O(h^\infty)+ 2C\iint_{t<2t_0}|\phi(s)||\Psi u_{h,\zeta}|^2dsdt\,.$$
We use the upper  bound for $\lambda_1(h,\zeta)$ in Proposition~\ref{prop:ub}, the lower bound for $q_{h,\zeta}(\cdot)$ in Lemma~\ref{lem:conc-gs} and the simple lower bound $V_{h,\zeta}-\lambda_1(h,\zeta)\geq \frac14h$ in $\{t(x)\geq 2h^\rho\}$, we get,
$$\iint_{t<2h^\rho} \left(h^{3/2}\phi(s)-2Ch^{3/2}\delta^2\phi(s)-Ch^{r^*}\right)|\Psi u_{h,\zeta}|^2dsdt\leq Ch^{r^*}\,.$$
We choose $\delta=\frac1{2\sqrt{C}}$ and get
$$\iint_{t<2t_0} \left(\frac12\phi(s)-Ch^{r^*-\frac32}\right)|\Psi u_{h,\zeta}|^2dsdt\leq Ch^{r^*}\,.$$
In particular, setting $\eta=\max(r^*-\frac32,\frac14)$, we write
$$\iint_{\substack{t<2h^{2\rho}\\ \phi(s)\geq 4Ch^\eta}} \phi(s)|\Psi u_{h,\zeta}|^2dsdt\leq C\,.$$ Selecting $\eta^*\in(0,\eta)$  finishes the proof of Theorem~\ref{thm:conc-gs}.
\end{proof}

\section{Analysis of Diamagnetism}\label{sec:dia-mag}

\subsection*{Proof of Theorem~\ref{thm:dia-mag}}

There exists a simple relationship between the eigenvalues in \eqref{eq:gse} and \eqref{eq:gse*}. This relationship is displayed as follows
$$\tilde\mu_1(\beta;H)=H^{2}\mu_1(h;b,\alpha,\gamma)\,,$$
where $h=H^{-1}$, $b=1$ and $\gamma=\beta H^{-1+\alpha}$.

The assumption in Theorem~\ref{thm:dia-mag} ensure that
\begin{itemize}
\item As $\beta\to-\infty$, the semi-classical parameter $h\to0_+$\,;
\item The parameter $\gamma$ is uniformly bounded, i.e. $\gamma=\mathcal O(1)$ as $\beta\to-\infty$.
\end{itemize}
The ground state energy $\mu_1(h;b,\alpha,\gamma)$ is estimated in Theorem~\ref{thm:gs-gse} for $\alpha<\frac12$ (more precisely, this is a consequence of  Theorem~\ref{thm:mag-Lap}).  This yields the estimates announced for $\tilde \mu_1(\beta;H)$ in 
Theorem~\ref{thm:dia-mag} for $\alpha<\frac12$.

For $\alpha\geq \frac12$, the estimates in Theorem~\ref{thm:dia-mag} follow from the following result proved in \cite{Ka-jmp}
$$\mu_1(h;n=1,\alpha,\gamma)=
\left\{
\begin{array}{ll}
\Theta(0) h+ho(1)&{\rm if~}\alpha>\frac12\,,\\
&\\
\Theta(\gamma)h+ho(1)&{\rm if~}\alpha=\frac12\,.
\end{array}
\right.
$$ 

\subsection*{Acknowledgments} The author is supported by  the research funding program of  the Lebanese University.

\end{document}